\documentclass[11pt,reqno]{amsart}
\usepackage[margin=3.6cm]{geometry}
\usepackage{graphicx}
\usepackage{amssymb}
\usepackage{bbm}
\usepackage{color}

\newcommand{\R}{{\mathbb R}}

 \renewcommand{\Re} {\mathrm{Re\,}}

\newtheorem{theorem}{Theorem}

\newtheorem{lemma}[theorem]{Lemma}

\theoremstyle{definition}

\newtheorem{remark}{Remark}

\def\reals{{\mathbb R}}

\def\p{\partial}
\def\half{\frac{1}{2}}
\def\tchi{\tilde{\chi}}

\def\tpsi{\tilde{\psi}}
\def\supp{\mathrm{supp}\,}
\def\O{{\mathcal O}}
\def\tzeta{\tilde{\zeta}}

\def\tA{\tilde{A}}

\begin{document}

\title[Small-scale mass of Neumann eigenfunctions]{Small-scale mass estimates for Neumann eigenfunctions:  piecewise smooth planar domains}

\author[H. Christianson]{Hans Christianson}
\address{Department of Mathematics, UNC Chapel Hill} \email{hans@math.unc.edu}

\author[J. Toth]{John  A. Toth}
\address{Department of Mathematics and
Statistics, McGill University, 805 Sherbrooke Str. West, Montr\'eal
QC H3A 2K6, Ca\-na\-da.} \email{jtoth@math.mcgill.ca}

\maketitle

\begin{abstract} 
Let $\Omega$ be a piecewise-smooth, bounded convex domain in $\R^2$ and consider $L^2$-normalized Neumann eigenfunctions $\phi_{\lambda}$ with eigenvalue $\lambda^2$. Our main result is a small-scale {\em non-concentration} estimate: We prove that for {\em any} $x_0 \in \overline{\Omega},$ (including boundary and corner points)  and any $\delta \in [0,1),$ 
$$ \| \phi_\lambda \|_{B(x_0,\lambda^{-\delta})\cap \Omega} = O(\lambda^{-\delta/2}).$$ 
The proof is a stationary vector field argument combined with a small scale induction argument.

\end{abstract}\ \\
\section{introduction}
In this paper, we consider Neumann eigenfunctions in planar domains and prove a non-concentration estimate on shrinking balls, including at boundary and corner points.  Let $\Omega \subset \R^2$ be a bounded, convex planar domain with boundary $\partial \Omega.$ We say that $\Omega$ is {\em piecewise smooth} if the boundary  $\partial \Omega = \cup_{j=1}^N \Gamma_j$ such that there exist defining functions $f_j \in C^{\infty}(\R^2; \R)$ with
$$\Gamma_j \subset \{ x \in \R^2; f_j(x) = 0, \,\,\, df_j (x) \neq 0 \}.$$

 We refer to the $\Gamma_j$'s as the boundary edges. We say that a piecewise-smooth $\Omega$ is a {\em domain with corners} if the
$\Gamma_j$'s are diffeomorphic to closed intervals with $\Gamma_j \cap \Gamma_{j+1} = c_j  \in \R^2; j=1,...,,N,$ such that at   $c_j = \Gamma_j \cap \Gamma_{j+1},$ 
$$ \text{rank} \, ( df_j(c_j), df_{j+1}(c_j) ) = 2.$$
We refer to ${\mathcal C}:= \{ c_j \}_{j=1}^{N}$ as the set of {\em corner points} and the rank condition on the defining functions at the $c_j$'s ensures that the boundary edges $\Gamma_j; j=1,...,N$ intersect at non-zero angles. We denote the angle at a corner $c_j$ by $\alpha_j.$

A fundamental issue regarding eigenfunctions involves their concentration properties (or lack thereof) on small  balls with radius that depends on the eigenvalue $\lambda^2$ as $\lambda \to \infty.$  

Let $(M,g)$ be a compact Riemannian manifold without boundary and $\phi_\lambda$  a Laplace eigenfunction with eigenvalue $\lambda^{2}$. Then, as pointed  out in \cite{So}, using the explicit asymptotic formula for the half-wave operator $e^{it \sqrt{- \Delta}}: C^{\infty}(M) \to C^{\infty}(M)$ it is not hard to prove that there exists $C_M>0$ such that 

\begin{equation} \label{soggebound}
\| \phi_\lambda \|_{L^2(B(r))}^2 = O(r) \| \phi_\lambda \|_{L^2(M)}^2, \quad  \forall r \geq C_M \lambda^{-1} \end{equation}\

We refer to estimates of the form (\ref{soggebound}) as {\em non-concentration} bounds.
The example of highest weight spherical harmonics  on the round sphere (see also Remark \ref{gaussian} below) shows that (\ref{soggebound}) is, in general, sharp. However, in certain cases, one expects improvements. For instance, in the case of surfaces with non-positive curvature, one can get logarithmic improvements  \cite{So} (see also \cite{Han, HR}).

Since the proof of (\ref{soggebound}) uses the wave parametrix in a crucial way, the extension to manifolds with boundary is  non-obvious since the behaviour of  the  wave operators near $\partial \Omega$ is much more complicated than in the boundaryless case.
The   main result of this paper is an
extension of the bounds in (\ref{soggebound}) to Neumann
eigenfunctions of a bounded piecewise-smooth, convex planar
domain. Our basic method of proof here is entirely stationary and uses
a Rellich commutator argument rather than wave methods.
 This stationary approach allows us to deal with points on {\em boundaries and corners} as well as interior points.  We stress that our result below holds right up to the boundary, including corner points.

In order to state the theorem, it is useful to switch to the  standard semiclassical notation with $h = \lambda^{-1}$ so that asymptotics are estimated as $h \to 0^+$.
\begin{theorem}
  \label{T:non-con}
  Let $\Omega \subset \reals^2$ be a piecewise $C^\infty$ bounded, convex
  domain 
  and consider the semiclassical Neumann eigenfunction problem:
  \begin{equation*}
    \begin{cases}
      -h^2 \Delta \phi_h(x) = \phi_h(x),  \quad x \in \Omega, \\
      \p_\nu \phi_h  |_{\p \Omega} = 0, \\
      \| \phi_h \|_{L^2 ( \Omega) } = 1,
    \end{cases}
  \end{equation*}
  where $\p_\nu$ is the outward pointing normal derivative.
Let 
  $p_0 \in
  \overline{\Omega}$ be a point in $\Omega$ or on the boundary (including corners).  Then
  for any $0 \leq \delta <1$, 
  
  \begin{equation}
    \| \phi_h \|^2_{L^2 ( B(p_0 , h^{\delta}) \cap \overline{\Omega})} = O(h^{\delta}).
  \end{equation}

\end{theorem}\

\begin{remark}
We stress that this result holds at any point, including corners where (possibly curved) transversal boundary components meet.

In the piecewise smooth case in the present paper, we assume the domain is convex so that eigenfunctions are in $H^2$.  Blowup asymptotics (at least for Dirichlet eigenfunctions) at a non-convex corner show eigenfunctions need not be globally in $H^2$ in general.

\end{remark}

\begin{remark}

In a forthcoming companion paper \cite{ChTo-non-con-smooth}, we further investigate non-concentration estimates at interior points, as well as analytic manifolds with analytic boundary.  Interestingly, if $\Omega$ is an analytic manifold with an analytic {\it concave} boundary, then eigenfunctions can be extended past the boundary and the method for interior points works.  An example of such a manifold is a compact analytic Riemann surface with several discs removed.  On the other hand, if the boundary is not concave, the extended eigenfunctions may exhibit too much growth to apply the interior method.

\end{remark}

\begin{remark}
The theorem is also true for Dirichlet eigenfunctions, but the proof
in that case is much easier.  We will point out the small
modifications necessary to the proof of Theorem \ref{T:non-con} in the proof.

\end{remark}
\begin{remark}
\label{R:h-delta}
  
  As the proof will indicate,
  the bound for eigenfunction $L^2$ mass in
  a ball of radius $h^\delta$, $0 \leq \delta \leq 1/2$, is relatively straightforward.
The cases where $0 \leq \delta <1/2$ follow immediately from the argument
proving the $\delta = 1/2$ case.  
  To
  improve to $1/2 < \delta < 1$, we use the estimate for $\delta =
  1/2$ to bootstrap to $\delta = 2/3$.  Then an induction step proves
  that for any integer $k>0$, the result is true for $\delta = 1-1/3k$.

\end{remark}

\begin{remark} \label{gaussian}  The estimate in Theorem \ref{T:non-con} is sharp, at least for $\delta = 1/2$. To see this,  let $\gamma \subset \Omega$ be a geodesic segment with $\gamma = \{ (x',x_n=0) \in \Omega; |x'| < \delta \}$ and $U = \{(x',x_n); |x_n| < \delta \}$ be a tubular neighbourhood, where $(x',x_n): U \to \R^n$ are Fermi coordinates. An $L^2$-normalized Gaussian beam localized on $\gamma$ is  of the form
$$ \phi_h(x) = (2\pi h)^{-1/4} e^{-x_n^2/h} \,e^{ i x'/h} \, ( \,  a(x',x_n;) + O(h) \,); \, a \in C^{\infty}(U), \,\, |a(x)| >0, \,\, x \in U.$$
It follows that
$$ \|\phi_h \|_{B(0,h^{1/2})}^2 \sim \int_{|x_n| \leq h^{1/2}} \int_{|x'| < h^{1/2}} |\phi_h(x)|^2 \, dx \sim h^{1/2}.$$

Consider the case where $\Omega = \{(x,y); \frac{x^2}{a^2} +
\frac{y^2}{b^2} = 1, y \geq 0 \}$  where $a>b>0$ is the half-ellipse
and let $\phi_h$ be an $L^2$-normalized Neumann eigenfunction. It is
well-known (see \cite{TZ1} section 2.2)  that there exists a subsequence of eigenfunctions that are Gaussian beams along the major axis $\{ (x,0); - a \leq x \leq a \}.$
Consequently, the estimate in Theorem \ref{T:non-con} is sharp in
general. In the special case where the $u_h$ satisfy polynomial
 small-scale quantum ergodicity (SSQE) on a scales $h^{1/2},$ since the volume of a ball of radius $h^{1/2}$ is
$h,$ one putatively expects a bound of $O(h)$ on the RHS in Theorem
\ref{T:non-con}. Unfortunately, to our knowledge, there are no
rigorous results  on polynomial SSQE known at present,
although logarithmic SSQE was proved by X. Han \cite{Han}. \end{remark}


 \section{One point non-concentration in shrinking balls}

Before jumping into the details of the proof of Theorem
\ref{T:non-con}, let us sketch the main intuitive idea.  The result for interior points follows easily from the result for boundary points, so we sketch the idea in the case of a point on a flat side; the analysis for a point on a curved side and at a corner point will be in the full proof below.  Suppose $p_0$
is a point on a flat side of $\p \Omega$, and assume for simplicity
that $\p \Omega = \{ y = 0 \}$ locally near $p_0$ and $p_0 = (0,0)$.
Let $\chi$ be a smooth monotone bounded function, $\chi(y) \sim h^{-1/2} y$
in an $h^{1/2}$ neighbourhood of $y = 0$, and constant outside a
neighbourhood of size $M h^{1/2}$ for large $M$.  Then $\chi'(y)$ is a
bump function supported on $-M h^{1/2} \leq y \leq M h^{1/2}$ with
$\chi'(y) \sim h^{-1/2}$ on $-h^{1/2} \leq y \leq h^{1/2}$.  We then
apply a Rellich commutator type argument:
\begin{align*}
  \int_\Omega & ([-h^2 \Delta -1, \chi \p_y ] \phi_h) \phi_h dV \\
  & = -2
  \int_\Omega \chi'(y) (h^2 \p_y^2 \phi_h ) \phi_h dV + \O(1) \\
  & = 2
  \int_\Omega \chi' | h \p_y \phi_h |^2 dV + \O(1) \\
  & \gtrapprox h^{-1/2} \int_{B((0,0),h^{1/2}) \cap \Omega} | h \p_y
  \phi_h |^2 dV - \O(1).
\end{align*}
Computing the commutator explicitly shows the left hand side is
bounded.  
Adding a similar computation with $\chi(x) \p_x$ and 
rearranging would prove the theorem (for $\delta = 1/2$).  
A suitable $h$-dependent cutoff allows us to integrate by parts to go
from estimating $\| h \nabla \phi_h \|_{L^2(B)}$ from below to
estimating $\| \phi_h \|_{L^2(B)}$ from below.  Here the 
$O(1)$ error term is from differentiating $\chi$ twice: $h \chi'' =
O(1)$.  This allows us to prove the theorem at the limiting  scale
$h^{1/2}$.  The tricky part is using the $\delta = 1/2$ result to 
prove the result for $\delta = 2/3$, and then applying an induction
argument to get the result for any $\delta<1$.   
Of course
in this little sketch, the $\O(1)$ terms from integrating by parts,
etc. are actually very subtle in the case of Neumann eigenfunctions,
and the bulk of the proof is dealing with these ``lower order terms''. 

%
\subsection{Notation and preliminaries}
\label{SS:preliminaries}

Let us first establish convenient coordinates for the proof.  In the case where $p_0$ is on a smooth side away from a corner, we
rotate, translate, and use graph coordinates  so that $\Gamma
\subset \{ y = \alpha(x) \}$ for  locally  smooth $\alpha$, $p_0 = (0,0)$, and $\Omega$
lies below the curve $y=\alpha(x)$.  We will eventually need to
invert $y = \alpha(x)$, so rotate further if necessary to assume that
$\alpha'(0) = 1$.  Let $\beta = \alpha^{-1}$ so that $y = \alpha(x)
\iff x = \beta(y)$ locally near $(0,0)$.  We assume as before that
$\Omega$ lies below the curve $y = \alpha(x)$; that is, locally $\Omega \subset \{ (x,y); y < \alpha(x) \}.$

Let $\kappa = (1 + (\alpha')^2)^\half$ be the arclength element with
respect to $x$.  Then the normal and tangential derivatives are respectively
\begin{equation}
  \label{E:nuxy-1}
\p_\nu = -\frac{\alpha'}{\kappa} \p_x + \frac{1}{\kappa} \p_y , \,\,\,
\p_\tau = \frac{1}{\kappa} \p_x + \frac{\alpha'}{\kappa} \p_y
\end{equation}
so that
\begin{equation}
  \label{E:nuxy-2}
\p_x = \frac{1}{\kappa} \p_\tau - \frac{\alpha'}{\kappa} \p_\nu ,
\,\,\, \p_y = \frac{\alpha'}{\kappa} \p_\tau + \frac{1}{\kappa} \p_\nu.
\end{equation}

Now in the case $p_0$ is a corner, 
translate and
rotate so that $p_0 = (0,0)$, and $\p \Omega$ locally has two smooth
sections.  That is, after a rotation and translation, there exist  locally  smooth functions $\alpha_1$ and
$\alpha_2$ such that $\alpha_1$ is monotone increasing, $\alpha_2$ is
monotone decreasing, $\alpha_1'(0) >0$, and $ \alpha_2'(0) <0$,
and  near $(0,0)$ 
\[
\p \Omega = \{ y = \alpha_1(x) ; 0 \leq x \leq \eta \} \cup \{ y = 
\alpha_2(x) ; 0 \leq x \leq \eta \}
\]
for some $\eta>0$ independent of $h$.  We assume further that locally $\Omega$ lies to
the right of these sections (this is automatic due to convexity of
$\Omega$).
Then locally each $\alpha_j$ has an inverse, which we denote
$\beta_j$.  That is, near $(0,0)$, $y = \alpha_j(x) \iff x =
\beta_j (y)$.

We will need to know the tangential and normal derivatives in these
coordinates.  For the top section where $y = \alpha_1(x)$, we have
already computed in \eqref{E:nuxy-1} and \eqref{E:nuxy-2} with
$\alpha$ replaced by $\alpha_1$.  For the
bottom section where $y = \alpha_2(x)$, let $\kappa_2 = (1 +
(\alpha_2')^2)^\half$ so that the tangent is $\tau = \kappa_2^{-1} ( 1
, \alpha_2')$.  Recalling that $\alpha_2' < 0$ near $0$, the outward
unit normal then is $\nu = \kappa_2^{-1}(\alpha_2' , -1 )$.  Hence
\begin{equation}
  \label{E:nuxy-3}
\p_\nu = \frac{\alpha_2'}{\kappa_2} \p_x - \frac{1}{\kappa_2} \p_y , \,\,\,
\p_\tau = \frac{1}{\kappa_2} \p_x + \frac{\alpha_2'}{\kappa_2} \p_y
\end{equation}
so that
\begin{equation}
  \label{E:nuxy-4}
\p_x = \frac{1}{\kappa_2} \p_\tau + \frac{\alpha_2'}{\kappa_2} \p_\nu ,
\,\,\, \p_y = \frac{\alpha_2'}{\kappa_2} \p_\tau - \frac{1}{\kappa_2} \p_\nu.
\end{equation}

In the following, we will employ a number of convenient spatial cutoff
functions that we introduce here.

  Let $\tchi(s) \in C^\infty ( \reals)$ satisfy the following conditions:
  \begin{itemize}

  \item $\tchi$ is odd,

  \item $\tchi' \geq 0$,

    \item $\tchi(s) \equiv -1$ for $s \leq -3$ and $\tchi(s) \equiv 1$ for $s
      \geq 3$,

    \item
      $\tchi(-1) = -1/2$ and $\tchi(1) = 1/2$,

    \item $\tchi(s) = \frac{s}{2}$ for $-1 \leq s \leq 1$.

  \end{itemize}
See Figure \ref{F:tchi-2} for a picture. 
    \begin{figure}
\hfill
\centerline{
\begingroup%
  \makeatletter%
  \providecommand\color[2][]{%
    \errmessage{(Inkscape) Color is used for the text in Inkscape, but the package 'color.sty' is not loaded}%
    \renewcommand\color[2][]{}%
  }%
  \providecommand\transparent[1]{%
    \errmessage{(Inkscape) Transparency is used (non-zero) for the text in Inkscape, but the package 'transparent.sty' is not loaded}%
    \renewcommand\transparent[1]{}%
  }%
  \providecommand\rotatebox[2]{#2}%
  \newcommand*\fsize{\dimexpr\f@size pt\relax}%
  \newcommand*\lineheight[1]{\fontsize{\fsize}{#1\fsize}\selectfont}%
  \ifx\svgwidth\undefined%
    \setlength{\unitlength}{273.47494125bp}%
    \ifx\svgscale\undefined%
      \relax%
    \else%
      \setlength{\unitlength}{\unitlength * \real{\svgscale}}%
    \fi%
  \else%
    \setlength{\unitlength}{\svgwidth}%
  \fi%
  \global\let\svgwidth\undefined%
  \global\let\svgscale\undefined%
  \makeatother%
  \begin{picture}(1,0.38988665)%
    \lineheight{1}%
    \setlength\tabcolsep{0pt}%
    \put(0,0){\includegraphics[width=\unitlength,page=1]{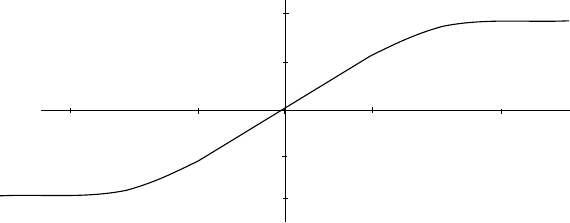}}%
    \put(0.09378875,0.16286937){\color[rgb]{0,0,0}\makebox(0,0)[lt]{\lineheight{1.25}\smash{\begin{tabular}[t]{l}$-3$\end{tabular}}}}%
    \put(0.32922772,0.16438834){\color[rgb]{0,0,0}\makebox(0,0)[lt]{\lineheight{1.25}\smash{\begin{tabular}[t]{l}$-1$\end{tabular}}}}%
    \put(0.64517173,0.15983148){\color[rgb]{0,0,0}\makebox(0,0)[lt]{\lineheight{1.25}\smash{\begin{tabular}[t]{l}$1$\end{tabular}}}}%
    \put(0.85782626,0.16135039){\color[rgb]{0,0,0}\makebox(0,0)[lt]{\lineheight{1.25}\smash{\begin{tabular}[t]{l}$3$\end{tabular}}}}%
  \end{picture}%
\endgroup%
}
\caption{\label{F:tchi-2} A  sketch of the function $\tchi$ used in
  the proof of Theorem \ref{T:non-con}.}
\hfill
\end{figure}
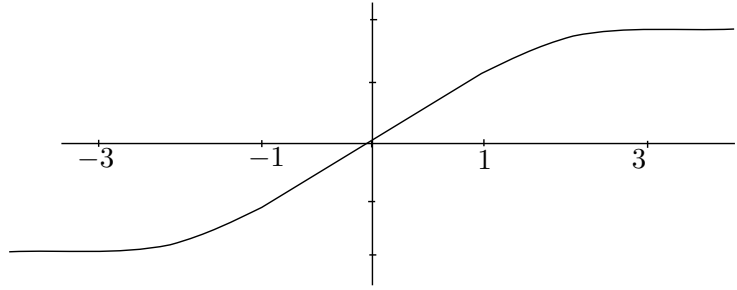
Let $\gamma(s) = \tchi'(s)$ so that $\gamma$ has support in $\{ -3
\leq s \leq 3 \}$, $\gamma(s) \geq 0$, and $\gamma(s) \equiv 1/2$ for
$| s | \leq 1$.

 In the following, it will also be  useful to define the corresponding semiclassically rescaled functions
 $$\tilde{\chi}_{\delta}(s):= \tilde{\chi} ( h^{-\delta} s), \quad \gamma_{\delta}(s):= \gamma( h^{-\delta} s), \,\,\, \delta \in [0,1].$$

    Choose also a smooth bump function $\tpsi(s) \in C^\infty (\reals)$
    satisfying
    \begin{itemize}
    \item $\tpsi(s)$ is even and $\tpsi' \leq 0$ for $s \geq 0$,

    \item $\tpsi(s) \equiv 1$ for $-1 \leq s \leq 1$,

    \item $\tpsi (s) \equiv 0$ for $| s | \geq 2$.
    \end{itemize}

In the sequel, we will need some a priori estimates on our eigenfunctions on the boundary.  This is summarized in the following Lemma, which follows from the work of Grisvard \cite{Grisvard-book} on convex domains combined with Sobolev embedding.

\begin{lemma}
\label{L:a-priori-bounded}
The Neumann eigenfunctions $\phi_h$ satisfy the following estimates:
\begin{enumerate}
\item $\phi_h \in H^2 ( \Omega)$;

\item $\phi_h \in C^\infty (\Omega)$, and $\phi_h \in C^\infty (\overline{\Omega} \setminus \mathcal{C})$ where $\mathcal{C}$ is the set of corner points as usual; and

\item for each $h>0$, there exists a constant $C_h$ such that $|\phi_h| \leq C_h$ on $\overline{\Omega}$.

\end{enumerate}

\end{lemma}
The meaning of the third assertion is that the eigenfunctions do not blow up at corners, even though the ``constant'' $C_h$ may be very large as $h \to 0$.  We will use this when integrating by parts along the boundary of $\Omega$.

We will also need a Sobolev type estimate in shrinking neighbourhoods of corners.  
    \begin{lemma}
    \label{L:Sobolev}
Let 
    $\eta \geq h$ and suppose $\zeta(x)$ is a smooth function satisfying $\supp \zeta \subset \{ |x| \leq \eta \}$ satisfying $\p_x^k \zeta = \O(\eta^{-k})$.  Then 
    \[
    \int_0^\eta \zeta(x) | \phi_h|^2 ( x , \alpha_1(x)) dx = \O(h^{-1}) \int_{\Omega \cap \{ |x| \leq 3 \eta \}} | \phi_h|^2 dV.
    \]
    
    \end{lemma}
    \begin{remark}
    The statement for the lower segment $y = \alpha_2(x)$ is similar.  
    Here the choice of $3 \eta$ is for convenience in the proof.  Any domain wider than $\eta$ will work as well.
    \end{remark}

    \begin{proof}
    The proof in the case of a smooth side follows from the more difficult proof in the case of a corner, so we will just prove the corner case.  In the coordinates above, we are interested in the boundary traces in a $\eta>0$ neighbourhood of $(0,0)$.  
%
Let 
    \[
    E(x) = \int_{0}^{\alpha_1} | \phi_h |^2 d y
    \]
and compute
\begin{align*}
E'(x) & = \alpha_1'(x)  |\phi_h |^2 ( x, \alpha_1(x)) + 2 \Re \int_{0}^{\alpha_1} (\p_x \phi_h) \bar{\phi}_h dy.
\end{align*}
From Lemma \ref{L:a-priori-bounded}, we know that $E(0) = 0$.  
Then if $\tzeta(x)$ is a smooth function such that $\tzeta(x) \equiv 0$ for $x \geq \eta$ and $\p_x^k \tzeta = \O(\eta^{-k})$,
we have
\begin{align}
\int_0^\eta \tzeta(x) E'(x) dx & = - \int_0^\eta \tzeta'(x) E(x) dx+ \tzeta(\eta) E(\eta) - \tzeta(0) E(0)  \notag \\
& = - \int_0^\eta \tzeta'(x) E(x) dx \notag \\
& = -\int_0^\eta \tzeta'(x) \int_{0}^{\alpha_1} | \phi_h |^2 dy dx  \notag\\
& = \O(\eta^{-1}) \int_0^\delta \int_{0}^{\alpha_1} |\phi_h |^2 dy dx. \label{E:E-prime-IBP}
\end{align}

On the other hand, using Cauchy's inequality with parameter, we have
\begin{align}
\int_0^\eta \tzeta(x) E'(x) dx & = \int_0^\eta \tzeta(x) (  \alpha_1'(x)  |\phi_h |^2 ( x, \alpha_1(x)) dx + 2 \Re \int_0^\eta \int_{0}^{\alpha_1}\tzeta(x)  (\p_x \phi_h) \bar{\phi}_h dy ) dx  \notag \\
& = \int_0^\eta \tzeta(x) \alpha_1'(x) | \phi_h|^2 dx + \O(1) \int_0^{\eta} \int_0^{\alpha_1} h^{-1} ( | h \p_x \phi_h|^2 + | \phi_h|^2 ) dy dx. \label{E:zeta-Cauchy}
\end{align}
   Rearranging \eqref{E:zeta-Cauchy}, we have
   \begin{align}
   \int_0^\eta \tzeta(x) \alpha_1'(x) | \phi_h|^2 (x, \alpha_1(x))dx & = \int_0^\eta \tzeta(x) E'(x) dx  + \O(h^{-1})  \int_0^{\eta} \int_0^{\alpha_1} ( | h \p_x \phi_h|^2 + | \phi_h|^2 ) dy dx \notag \\
   & = \O(h^{-1})  \int_0^{\eta} \int_0^{\alpha_1} ( | h \p_x \phi_h|^2 + | \phi_h|^2 ) dy dx 
      \label{E:zeta-bdy}
   \end{align}
   from \eqref{E:E-prime-IBP} since $\eta \geq h$.
   
   Now let $\Gamma(x)$ be a smooth bump function with $\Gamma(x) \equiv 1$ for $|x| \leq \eta$ with support in $\{ |x| \leq 2 \eta\}$ and $\p_x^k \Gamma = \O( \eta^{-k})$.  Let $\tilde{\Gamma}$ be a smooth bump function such that $\tilde{\Gamma} \equiv 1$ for $| x | \leq 2 \eta$ with support in $\{ | x | \leq 3 \eta \}$ and $\p_x^k \tilde{\Gamma} = \O( \eta^{-k})$. 
       Let 
    \[
    I = \int_\Omega \Gamma^2(x)  | h \p_x \phi_h|^2 dV 
    \] 
    so we have
   \begin{align}
    \int_0^{\eta} \int_0^{\alpha_1} &   | h \p_x \phi_h|^2 dy dx  \notag \\
    & \leq I \notag \\
    & \leq \int_\Omega \Gamma^2 (x) ( | h \p_x \phi_h|^2 + | h \p_x \phi_h |^2 ) dV \notag \\
    & = \int_\Omega \Gamma^2  (x) ( -h^2 \Delta \phi_h ) \bar{\phi}_h dV \notag \\
    & \quad - 2 \int_\Omega h \Gamma'(x) \Gamma(x) (h \p_x \phi_h) \bar{\phi}_h dV. \label{E:Gamma-prime}
    \end{align}
    The last term in \eqref{E:Gamma-prime} is estimated using Cauchy's inequality with small parameter $c>0$ independent of $h$ and $\eta$:
    \begin{align*}
    \Bigg| \int_\Omega & h \Gamma'(x) \Gamma(x) (h \p_x \phi_h) \bar{\phi}_h dV \Bigg| \\
    & \leq C \frac{h}{\eta} \int_\Omega \tilde{\Gamma}| \Gamma (x) h \p_x \phi_h|| \phi_h | dV \\ 
    & = C \int_\Omega | \Gamma h \p_x \phi_h | | \frac{h}{\eta} \tilde{\Gamma} \phi_h | dV \\ 
    & \leq C \int_\Omega (c \Gamma^2 | h \p_x \phi_h|^2 )dV + C \int_\Omega  (c^{-1} \tilde{\Gamma}^2 | \frac{h}{\eta} \phi_h|^2 )dV \\
    & \leq \half I + C\frac{h^2}{\eta^2}\int_\Omega \tilde{\Gamma}^2 | \phi_h|^2 dV
    \end{align*}
    for $c>0$ sufficiently small, but independent of $h$ and $\eta$.  
    This gives the estimate
    \begin{align*}
    I & \leq \int_\Omega \Gamma^2  (x) ( -h^2 \Delta \phi_h ) \bar{\phi}_h dV +  \half I + C\frac{h^2}{\eta^2}\int_\Omega \tilde{\Gamma}^2 | \phi_h|^2 dV \\
    & \leq C \int_\Omega \tilde{\Gamma}^2 | \phi_h |^2 dV + \half I,
    \end{align*}
    where we have used the eigenfunction equation and $h \leq \eta$.  
    Solving for $I$ gives
    \[
    I \leq C \int_\Omega \tilde{\Gamma}^2 | \phi_h|^2 dV.
    \]
    Plugging this estimate into the right hand side of \eqref{E:zeta-bdy}, we have
    \begin{align}
     \int_0^\delta \tzeta(x) \alpha_1'(x) | \phi_h|^2 dx & = \O(h^{-1})  \int_0^{\delta} \int_0^{\alpha_1} ( | h \p_x \phi_h|^2 + | \phi_h|^2 ) dy dx \notag \\
     & = \O(h^{-1}) \int_\Omega \tilde{\Gamma}^2 | \phi_h |^2 dV.
      \label{E:zeta-bdy2}
      \end{align}
%
%
   Finally, recall that $\alpha_1'(x)>0$ is bounded away from $0$ in a neighbourhood of $x = 0$, so that $\tzeta = \zeta/\alpha_1'$ is a smooth function satisfying the correct properties.  This completes the proof.

    \end{proof}


\subsection{The piecewise-smooth case: Proof of Theorem \ref{T:non-con}}
\begin{proof}[Proof of Theorem \ref{T:non-con}]

The proof will proceed by looking at smooth (not necessarily flat) boundary pieces away from corners
and at corners separately, although the proof for corners has much in
common with smooth sides.


 The proof has several steps.  First we establish the result for
$\delta = 1/2$.  The proof for $0 \leq \delta < 1/2$ is similar (and easier), so we
omit the details.  Then we use the $\delta = 1/2$ estimate to
bootstrap the $\delta = 2/3$ estimate.  Again, for $1/2 < \delta <
2/3$, the proof is the same as for $\delta = 2/3$ (but again easier).
Our final step is an induction to prove that for any integer $k>0$ the
result is true for $\delta = 1-1/3k$.


 \subsubsection{  Analysis away from corner points } 
    
    We first consider a boundary point $p_0$ which is on a smooth (not necessarily flat)
    component of the boundary 
  $\Gamma$ away from corners.  We continue to work in coordinates from Subsection \ref{SS:preliminaries}.
%
%
%
%
%

For $\epsilon>0$ sufficiently small but independent of $h$, let

\begin{equation}
  \chi(x,y) = \tchi ( x/h^{1/2}) \tpsi (x/\epsilon) \tpsi(
  y/\epsilon)\label{E:chidef}.
\end{equation}\

If $\epsilon >0$
is sufficiently small, we may assume that $\supp ( \chi |_{\p \Omega}) \subset
\Gamma$.  We have $\chi(x,y) = x/2h^{1/2}$ for $-h^{1/2} \leq x \leq
h^{1/2}$ and $-\epsilon \leq y \leq \epsilon$.  We use  the short hand
notation $\chi_x  := \partial_{x} \chi,  \chi_y := \partial_y \chi$, so    $\supp \chi_x$  consists of three connected components, one
near zero, one near $-\epsilon$, and one near $\epsilon$.  Note:
since $\tchi(x/h^{1/2})$ is constant for $ x  \leq -3 h^{1/2}$ and $x \geq
3h^{1/2}$, we have that  
$\chi_x$ depends on $h$ for $-3h^{1/2} \leq x \leq 3 h^{1/2}$, but on
the set 
$\{ | x | \geq \epsilon \}$, $\chi_x = \epsilon^{-1} \tchi ( x/h^{1/2}
) \tpsi' ( x/\epsilon)\tpsi(y/\epsilon) = \pm \epsilon^{-1} \tpsi' ( x/\epsilon)\tpsi(y/\epsilon)$ is independent of $h$.  This means that
\begin{equation}
  \label{E:chixsupp}
\chi_x(x,y) \geq  h^{-1/2} \gamma(x/h^{1/2}) \gamma(y/h^{1/2}) - \O(1)
\end{equation}
so that, in particular, $\chi_x \geq  h^{-1/2}/4$ on $B((0,0), h^{1/2})$.

In order to ease notation, let $r>0$ be a small parameter  not depending on $h$  such that
$r \gg \epsilon$ but a $r$ neighbourhood of $(0,0)$ still
does not meet any  corners.  This is just so that integrating in
$[-r ,r]^2 \cap \Omega$ includes the
full support of $\chi$ inside $\Omega$.   See Figure
\ref{F:smooth-side-chi} for a picture.

    \begin{figure}
\hfill
\centerline{
\begingroup%
  \makeatletter%
  \providecommand\color[2][]{%
    \errmessage{(Inkscape) Color is used for the text in Inkscape, but the package 'color.sty' is not loaded}%
    \renewcommand\color[2][]{}%
  }%
  \providecommand\transparent[1]{%
    \errmessage{(Inkscape) Transparency is used (non-zero) for the text in Inkscape, but the package 'transparent.sty' is not loaded}%
    \renewcommand\transparent[1]{}%
  }%
  \providecommand\rotatebox[2]{#2}%
  \newcommand*\fsize{\dimexpr\f@size pt\relax}%
  \newcommand*\lineheight[1]{\fontsize{\fsize}{#1\fsize}\selectfont}%
  \ifx\svgwidth\undefined%
    \setlength{\unitlength}{232.93729591bp}%
    \ifx\svgscale\undefined%
      \relax%
    \else%
      \setlength{\unitlength}{\unitlength * \real{\svgscale}}%
    \fi%
  \else%
    \setlength{\unitlength}{\svgwidth}%
  \fi%
  \global\let\svgwidth\undefined%
  \global\let\svgscale\undefined%
  \makeatother%
  \begin{picture}(1,0.76970696)%
    \lineheight{1}%
    \setlength\tabcolsep{0pt}%
    \put(0,0){\includegraphics[width=\unitlength,page=1]{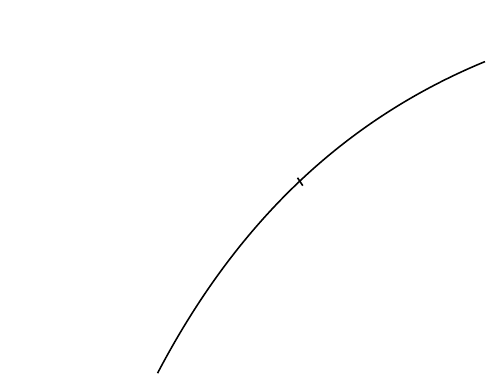}}%
    \put(0.53146402,0.41244644){\color[rgb]{0,0,0}\makebox(0,0)[lt]{\lineheight{1.25}\smash{\begin{tabular}[t]{l}$(0,0)$\end{tabular}}}}%
    \put(0.40044624,0.07339991){\color[rgb]{0,0,0}\makebox(0,0)[lt]{\lineheight{1.25}\smash{\begin{tabular}[t]{l}$y = \alpha(x)$\end{tabular}}}}%
    \put(0,0){\includegraphics[width=\unitlength,page=2]{smooth-side-chi.pdf}}%
    \put(0.1632925,0.73709022){\color[rgb]{0,0,0}\makebox(0,0)[lt]{\lineheight{1.25}\smash{\begin{tabular}[t]{l}$\chi \equiv \tchi(x/h^\half)$\end{tabular}}}}%
    \put(-0.00356352,0.17667156){\color[rgb]{0,0,0}\makebox(0,0)[lt]{\lineheight{1.25}\smash{\begin{tabular}[t]{l}$\supp(\chi)$\end{tabular}}}}%
  \end{picture}%
\endgroup%
}
\caption{\label{F:smooth-side-chi} $\Omega$ in a neighbourhood of a
  point on a smooth side and the function $\chi$.}
\hfill
\end{figure}

We will use a Rellich-type commutator argument, but terms that appear
``lower order'' have non-trivial dependence on $h$ and are not really
lower order.  We have
\[
[-h^2 \Delta , \chi \p_x ] = -2 \chi_x h^2 \p_x^2 - h \chi_{xx} h \p_x
- 2 \chi_y h \p_x h \p_y - h\chi_{yy} h \p_x,
\]
so integrating over $\Omega$ we have
\begin{align}
  \int_{\Omega} & ([-h^2 \Delta -1, \chi \p_x ] \phi_h ) \phi_h dV
  \notag \\
  & = \int_{\Omega} ((-2 \chi_x h^2 \p_x^2 - h \chi_{xx} h \p_x
  - 2 \chi_y h \p_x h \p_y - h\chi_{yy} h \p_x) \phi_h ) \phi_h dV.
  \label{E:first-comm}
\end{align}

We recall the standard identity for first order derivatives:
\[
\int_\Omega | h \nabla \phi_h |^2 dV = \int_\Omega (-h^2 \Delta \phi_h
) \phi_h dV = 1,
\]
which we use to estimate the lower order terms.  
Since $\chi_{xx} = \O(h^{-1})$ and $\chi_y$ and $\chi_{yy}$ are
bounded independent of $h$, we have
\[
\left| \int_\Omega (h \chi_{xx} h \p_x \phi_h ) \phi_h dV \right| \leq C \int ( | h \p_x \phi_h|^2 + | \phi_h|^2) dV = \O(1),
\]
and similarly
\[
\left| \int_\Omega (h \chi_{yy} h \p_y \phi_h ) \phi_h dV \right| = \O(1).
\]
For the mixed derivative term, 
we have
\begin{align*}
  \int_\Omega&  \chi_y (h \p_x h \p_y \phi_h ) \phi_h dV \\
  & = \int_{-r}^r \int_{-r}^{\alpha(x)} \chi_y (h \p_y h \p_x \phi_h )
  \phi_h dy dx\\
  & = - \int_{-r}^r \int_{-r}^{\alpha(x)}( h \p_x \phi_h) (h \chi_{yy}
  \phi_h + \chi_y h \p_y \phi_h ) dy dx \\
  & \quad + \int_{-r}^r h \chi_y ( h \p_x \phi_h) \phi_h
  |_{-r}^{\alpha(x)} dx.
\end{align*}
For the boundary term, the support properties of $\chi$ means
\[
\int_{-r}^r h \chi_y ( h \p_x \phi_h) \phi_h
  |_{-r}^{\alpha(x)} dx = \int_{-r}^r h \chi_y ( h \p_x \phi_h) \phi_h
  (x,{\alpha(x)}) dx.
  \]
  Along the face $y = \alpha(x)$, we have $h \p_x \phi_h = \kappa^{-1}
  h \p_\tau \phi_h$, so, in tangent coordinates,
  \begin{align}
  \int_{-r}^r & h \chi_y ( h \p_x \phi_h) \phi_h
  (x,{\alpha(x)}) dx \notag \\
&  = \int h \chi_y \kappa^{-1} (h \p_\tau \phi_h)
  \phi_h dS \notag \\
  & = \frac{h^2}{2} \int \chi_y \kappa^{-1} \p_\tau | \phi_h|^2 dS
  \notag \\
  & = - \frac{h^2}{2} \int (\p_\tau \chi_y \kappa^{-1}) | \phi_h |^2
  dS . \label{E:bdy-h-Sobolev-1}
  \end{align}
  The function $\p_\tau \chi_y \kappa^{-1} = \O(h^{-1/2})$, so, using
  the standard $h$-Sobolev estimates,
  \begin{align}
& \left| \frac{h^2}{2} \int (\p_\tau \chi_y \kappa^{-1}) | \phi_h |^2
    dS \right| \notag \\
    & \quad = \O(h^{1/2}) \int_\Omega (| h \nabla \phi_h |^2 + | \phi_h |^2 dV)
    \notag \\
    & \quad = \O(h^{1/2}). \label{E:bdy-h-Sobolev-2}
  \end{align}
  This implies
  \[
  \int_\Omega  \chi_y (h \p_x h \p_y \phi_h ) \phi_h dV = \O(1),
  \]
so that  \eqref{E:first-comm} becomes
  \begin{align*}
      \int_{\Omega} & ([-h^2 \Delta -1, \chi \p_x ] \phi_h ) \phi_h dV
  \\
   & = -2 \int_\Omega (\chi_x h^2 \p_x^2 \phi_h ) \phi_h dV + \O(1).
  \end{align*}

Now inside $\Omega$,  $\supp \chi_x \subset \{ (x,y); \beta(y) < x < r,  |y| < r \},$  by an integration by parts, 
\begin{align}
  -2& \int_\Omega (\chi_x h^2 \p_x^2 \phi_h ) \phi_h dV \notag \\
  & = -2 \int_{-r}^r \int_{\beta(y)}^r (\chi_x h^2 \p_x^2
  \phi_h ) \phi_h dx dy \notag \\
  & = 2\int_{-r}^r \int_{\beta(y)}^r (\chi_x h \p_x
  \phi_h ) h \p_x \phi_h dx dy \notag \\
  & \quad +  2\int_{-r}^r \int_{\beta(y)}^r h (\chi_{xx} h \p_x
  \phi_h )  \phi_h dx dy \notag \\
  & \quad - 2 \int_{-r}^r \left( (h\chi_x h \p_x \phi_h)
  \phi_h\right) |^r_{\beta(y)} dy \notag \\
  & = 2\int_{-r}^r \int_{\beta(y)}^r \chi_x |h \p_x
  \phi_h|^2 dx dy\notag  \\
  & \quad - 2 \int_{-r}^r \left( h(\chi_x h \p_x \phi_h)
  \phi_h\right) |_{\beta(y)}^r dy + \O(1), \label{E:boundaryI1}
\end{align}
where we have again used that $\chi_{xx} = \O(h^{-1})$.
Unfortunately, as $\chi_x = \O(h^{-1/2})$, the boundary term is not
necessarily bounded  in the Neumann case. 

 However, we will see that the largest part cancels with
a similar boundary term when we run a similar argument for a vector field in
the $\p_y$ direction.  Let
\[
I_1 = - 2 \int_{-r}^r \left( (h\chi_x h \p_x \phi_h)
\phi_h\right) |_{x=\beta(y)}^{x=r} dy
\]
be the boundary term from \eqref{E:boundaryI1}.  Using the support
properties of $\chi_x$, we have $\chi_x ( r , y) = 0$, so that
\[
I_1  = 2 \int_{-r}^r \left( (h\chi_x h \p_x \phi_h)
\phi_h\right) (\beta(y), y) dy.
\]
We now change variables $y = \alpha(x)$ so that
\begin{equation}
  \label{E:bdyx1}
I_1 = 2 \int_{-r}^r \left( (  h  \chi_x h \p_x \phi_h)
\phi_h\right) (x, \alpha(x)) \alpha' dx.
\end{equation}
We will return to this shortly.

Consider now the function
\begin{equation} \label{rho}
\rho(x,y) := \alpha'(x) \tchi ( \beta(y)/ h^{1/2}) \tpsi (x/\epsilon) \tpsi(
y/\epsilon). \end{equation}
We have
\[
[-h^2 \Delta -1, \rho \p_y ] = -2 \rho_y h^2 \p_y^2 - h \rho_{yy} h
\p_y - 2 \rho_{x} h \p_y h \p_x - h\rho_{xx} h \p_y.
\]
Again, since $\rho_{yy} = \O(h^{-1})$ and $\rho_x$ and $\rho_{xx}$ are
bounded, we have
\[
\int_\Omega ([-h^2 \Delta -1 , \rho \p_y] \phi_h ) \phi_h dV = -2
\int_\Omega ( \rho_y h^2 \p_y^2 \phi_h) \phi_h dV + \O(1).
\]
We again integrate by parts, but  now  in the $y$ direction.  We have
\begin{align}
  -2 &
  \int_\Omega ( \rho_y h^2 \p_y^2 \phi_h) \phi_h dV \notag \\
  & = -2 \int_{-r}^r \int_{-r}^{\alpha(x)} (\rho_y h^2
  \p_y^2 \phi_h ) \phi_h dy dx \notag \\
  & = 2\int_{-r}^r \int_{-r}^{\alpha(x)} \rho_y| h
  \p_y \phi_h |^2 dy dx \notag \\
  & \quad + 2
\int_{-r}^r \int_{-r}^{\alpha(x)} (h \rho_{yy} h
\p_y \phi_h ) \phi_h dy dx
\notag \\
& \quad -2 \int_{-r}^r \left( (h \rho_y h \p_y \phi_h)
\phi_h \right) |_{-r}^{\alpha(x)}dx \notag \\
& = 2\int_{-r}^r \int_{-r}^{\alpha(x)} (\rho_y h
  \p_y \phi_h ) h \p_y \phi_h dy dx \notag \\
& \quad -2 \int_{-r}^r \left( (h \rho_y h \p_y \phi_h)
  \phi_h \right) (x,\alpha(x)) dx + \O(1). \label{E:boundaryI2}
\end{align}
Here we have again used that $\rho_{yy} = \O(h^{-1})$ and that
$\rho_y(x , -r) = 0$.

Now let
\begin{equation}
  \label{E:I2-1}
I_2 = -2 \int_{-r}^r \left( (h \rho_y h \p_y \phi_h)
\phi_h \right) (x,\alpha(x)) dx
\end{equation}
be the boundary term from \eqref{E:boundaryI2}.  We observe that
\begin{align*}
  \rho_y & = \alpha'(x) \beta'(y) h^{-1/2} \tchi'(\beta(y)/h^{1/2})
  \tpsi(x/\epsilon) \tpsi(y/\epsilon) + \O(1).
\end{align*}
In \eqref{E:boundaryI2}, we are evaluating at $y = \alpha(x)$, so we
get 
\begin{align}
\rho_y ( x, \alpha(x)) & =   \alpha'(x) \beta'(\alpha(x)) h^{-1/2} \tchi'(x/h^{1/2})
\tpsi(x/\epsilon) \tpsi(\alpha(x)/\epsilon) + \O(1) \notag \\
& = \chi_x (x , \alpha(x)) + \O(1) \label{E:rhoy-chix}
\end{align}
with $\chi$ as in \eqref{E:chidef}.  Substituting into
\eqref{E:I2-1}, we have
\[
I_2 = 
-2 \int_{-r}^r \left( (h \chi_x h \p_y \phi_h)
\phi_h \right) (x,\alpha(x)) dx + \O(1).
\]

We now use the Neumann boundary conditions.  We have 
\begin{align}   0 & = \p_\nu \phi_h(x , \alpha(x)) &
  \notag \\
  & = -\frac{\alpha'}{\kappa} \p_x \phi_h ( x, \alpha(x)) +
  \frac{1}{\kappa} \p_y \phi_h(x, \alpha(x)) \label{E:px-py}
\end{align}
so that $\alpha' \p_x \phi_h (x, \alpha(x)) = \p_y \phi_h ( x ,
\alpha(x)).$  Substituting into \eqref{E:bdyx1}, we have

\begin{align} \label{display1}
  I_1 + I_2 & = \nonumber
2 \int_{-r}^r \left( (h\chi_x h \p_x \phi_h)
\phi_h\right) (x, \alpha(x)) \alpha' dx \\ \nonumber
& \quad 
-2 \int_{-r}^r \left( (h \rho_y h \p_y \phi_h)
\phi_h \right) (x,\alpha(x)) dx  \\ \nonumber
& = 
2 \int_{-r}^r \left( (  h \chi_x h \p_x \phi_h)
\phi_h\right) (x, \alpha(x)) \alpha' dx \\ \nonumber
& \quad 
-2 \int_{-r}^r \left( (h \chi_x h \p_y \phi_h)
\phi_h \right) (x,\alpha(x)) dx + \O(1) \\
& = \O(1).
\end{align}

Summing \eqref{E:boundaryI1} and \eqref{E:boundaryI2} we have
\begin{align*}
  \int_\Omega & ([-h^2 \Delta -1, \chi \p_x ] \phi_h) \phi_h dV 
  \\
  & \quad + \int_\Omega  ([-h^2 \Delta -1, \rho \p_y ] \phi_h) \phi_h dV \\
  & = 2\int_\Omega \chi_x | h \p_x \phi_h |^2 dV + 2 \int_\Omega
   \rho_y | h \p_y \phi_h |^2 dV + \O(1).
  \end{align*}
From \eqref{E:chixsupp} we have 
\[
\chi_x \geq  h^{-1/2} \gamma(x/h^{1/2}) \gamma(y/h^{1/2}) -
\O(1),
\]
and similarly there is a constant $c_0 >0$ independent of $h$ such that
\[
\rho_y \geq   c_0 h^{-1/2} \gamma(x/h^{1/2}) \gamma(y/h^{1/2}) -
\O(1).
\]
Let $c_1 = \min (1, c_0)$ so that	
\begin{align} 
    \int_\Omega & ([-h^2 \Delta -1, \chi \p_x ] \phi_h) \phi_h dV 
  \notag \\
  & \quad + \int_\Omega  ([-h^2 \Delta -1, \rho \p_y ] \phi_h) \phi_h
  dV \notag  \\
  & \geq c_1 \int_\Omega  h^{-1/2} \gamma(x/h^{1/2})
  \gamma(y/h^{1/2}) ( | h \p_x \phi_h|^2 + | h \p_y \phi_h |^2) dV -
  \O(1) \notag \\
  & = c_1 \int_\Omega  h^{-1/2} \gamma(x/h^{1/2})
  \gamma(y/h^{1/2}) ( - h^2 \p_x^2 \phi_h - h^2 \p_y^2 \phi_h )\phi_h
  dV \notag \\
  & \quad + c_1 h\int_{\p \Omega} h^{-1/2} \gamma(x/h^{1/2})
  \gamma(y/h^{1/2}) (h \p_\nu \phi_h) \phi_h dS
  -
  \O(1) 
 \notag  \\
  & = c_1 \int_\Omega  h^{-1/2} \gamma(x/h^{1/2})
  \gamma(y/h^{1/2}) | \phi_h|^2 
  dV - \O(1), \label{E:comm-sum-1}
\end{align}
where,  in the last line of  (\ref{E:comm-sum-1}),  we have used the eigenfunction equation and the Neumann boundary
conditions.  Since
\begin{equation*}
c_1 \int_\Omega  h^{-1/2} \gamma(x/h^{1/2})
  \gamma(y/h^{1/2}) | \phi_h|^2 
  dV \geq \frac{c_1}{4} \int_{B(p_0 , h^{1/2})}  h^{-1/2}  | \phi_h|^2 
  dV,
  \end{equation*}
  we have
  \begin{align}
    \frac{c_1}{4} \int_{B(p_0 , h^{1/2})}  h^{-1/2}  | \phi_h|^2 
    dV
    & \leq 
\int_\Omega  ([-h^2 \Delta -1, \chi \p_x ] \phi_h) \phi_h dV 
  \notag \\
  & \quad + \int_\Omega  ([-h^2 \Delta -1, \rho \p_y ] \phi_h) \phi_h dV
  + \O(1). \label{E:Oh-half}
  \end{align}

  On the other hand, expanding the commutator, using the eigenfunction equation, and integrating by parts, we have
  \begin{align}
    \int_\Omega  & ([-h^2 \Delta -1, \chi \p_x ] \phi_h) \phi_h dV
    \notag \\
    & = \int_\Omega ((-h^2 \Delta -1) \chi \p_x \phi_h) \phi_h dV  -
    \int_\Omega (\chi \p_x (-h^2 \Delta -1)  \phi_h) \phi_h dV \notag \\
    & = \int_\Omega (\chi \p_x \phi_h ) ((-h^2 \Delta -1) \phi_h) dV
     - \int_{\p \Omega} (h \p_\nu \chi h \p_x \phi_h ) \phi_h dS
    \notag \\
    & \quad + \int_{\p \Omega} ( \chi h \p_x \phi_h) ( h \p_\nu
    \phi_h) dS \notag \\
    & = - \int_{\p \Omega} (h \p_\nu \chi h \p_x \phi_h ) \phi_h dS.\label{E:comm-exp-111}
  \end{align}
  Using \eqref{E:nuxy-1}, \eqref{E:nuxy-2}, and the Neumann boundary
  conditions, we have
  \begin{align}
   h \p_\nu \chi h \p_x \phi_h & = \left( -\frac{\alpha'}{\kappa} h\p_x
   + \frac{1}{\kappa} h\p_y \right) \chi h \p_x \phi_h \notag \\
   & = \left( - \frac{ \alpha'}{\kappa} h \chi_x + \frac{1}{\kappa} h
   \chi_y \right)h \p_x \phi_h  \notag \\
   & \quad + \chi h \p_\nu h \p_x \phi_h \notag \\
   & = \left( - \frac{ \alpha'}{\kappa} h \chi_x + \frac{1}{\kappa} h
   \chi_y \right)\left( \frac{1}{\kappa}\right) h \p_\tau \phi_h\notag
   \\
   & \quad - \chi \frac{ \alpha'}{\kappa} h^2 \p_\nu^2 \phi_h + \O(h)
   h \p_\tau \phi_h \notag \\
   & = - \frac{ \alpha'}{\kappa^2} h \chi_xh \p_\tau \phi_h- \chi \frac{
     \alpha'}{\kappa} h^2 \p_\nu^2 \phi_h + \O(h) h \p_\tau \phi_h. \label{E:boundary-102}
  \end{align}
  \begin{remark}
Here is where we see that dealing with the boundary terms for the
Neumann eigenfunctions is 
significantly more difficult than in the case of Dirichlet
eigenfunctions.  Indeed, in the easier case of Dirichlet eigenfunctions, since $\Omega$ is convex, $\int_{\p \Omega} |
h \p_\nu \phi_h |^2 dS$ is bounded and the integrand has a sign. 

    \end{remark}

  Plugging
\eqref{E:boundary-102} into \eqref{E:comm-exp-111}    and using
integration by parts and Sobolev embedding for the $O(h)$ terms
as we did in \eqref{E:bdy-h-Sobolev-1}-\eqref{E:bdy-h-Sobolev-2},
 we have
  \begin{align*}
       \int_\Omega  & ([-h^2 \Delta -1, \chi \p_x ] \phi_h) \phi_h dV
       \\
       & = \int_{\p \Omega} \left(\chi \frac{
         \alpha'}{\kappa} h^2 \p_\nu^2 \phi_h \right) \phi_h dS \\
       & \quad + \int_{\p \Omega} \left( \frac{ \alpha'}{\kappa^2} h
       \chi_xh \p_\tau \phi_h \right) \phi_h dS + \O(1).
  \end{align*}

  A similar computation gives
  \begin{align*}
    \int_\Omega & ([-h^2 \Delta -1, \rho \p_y ] \phi_h ) \phi_h dV \\
    & = - \int_{\p \Omega} \left( \rho \frac{1}{\kappa} h^2 \p_\nu^2
    \phi_h \right) \phi_h dS \\
    & \quad - \int_{\p \Omega} \left( \frac{\alpha'}{\kappa^2} h\rho_y h
    \p_\tau \phi_h \right)\phi_h dS  + \O(1).
  \end{align*}
  Recalling that
  \[
  \rho(x, \alpha(x)) = \alpha' \chi(x, \alpha(x))
  \]
  and
  \[
  \rho_y(x, \alpha(x)) = \chi_x (x , \alpha(x)) + \O(1),
  \]
  we sum:
  \begin{align}
    \int_\Omega  & ([-h^2 \Delta -1, \chi \p_x ] \phi_h) \phi_h dV
    \notag \\
    & \quad + \int_\Omega ([-h^2 \Delta -1, \rho \p_y ] \phi_h )
    \phi_h dV \notag \\
    & = \int_{\p \Omega} \left(\chi \frac{
         \alpha'}{\kappa} h^2 \p_\nu^2 \phi_h \right) \phi_h dS
    \notag \\
       & \quad + \int_{\p \Omega} \left( \frac{ \alpha'}{\kappa^2} h
    \chi_xh \p_\tau \phi_h \right) \phi_h dS \notag \\
    & \quad - \int_{\p \Omega} \left( \rho \frac{1}{\kappa} h^2 \p_\nu^2
    \phi_h \right) \phi_h dS \notag \\
    & \quad - \int_{\p \Omega} \left( \frac{\alpha'}{\kappa^2} h\rho_y h
    \p_\tau \phi_h \right)\phi_h dS  + \O(1) \notag \\
    & = \O(1), \label{E:O1}
  \end{align}
  since the displayed terms on the RHS of (\ref{E:O1}) all cancel.

  Finally, equating \eqref{E:Oh-half} with \eqref{E:O1}, we get
  \[
  \int_{B(p_0 , h^{1/2})}  h^{-1/2}  | \phi_h|^2 
    dV = \O(1)
    \]
    as asserted.

\subsubsection{   Analysis near corner points }  We now consider the case where $p_0$ is a corner.

For $\epsilon>0$ sufficiently small, let $\chi(x,y)$ be the same as in
\eqref{E:chidef}.  We again use a parameter $r \gg \epsilon$ but
sufficiently small that $[-r, r]^2$ does not meet any other
corners.  Again, this is just to ease notation in our integral
expressions.  
Applying the same commutator argument as in the smooth boundary segment case, the
interior computations are the same, we just need to check what happens
on the boundary.  The key difference from the  case with no corners is
that boundary integrals have to be considered piecewise.  See Figure
\ref{F:corner-chi} for a picture of the setup.

   \begin{figure}
\hfill
\centerline{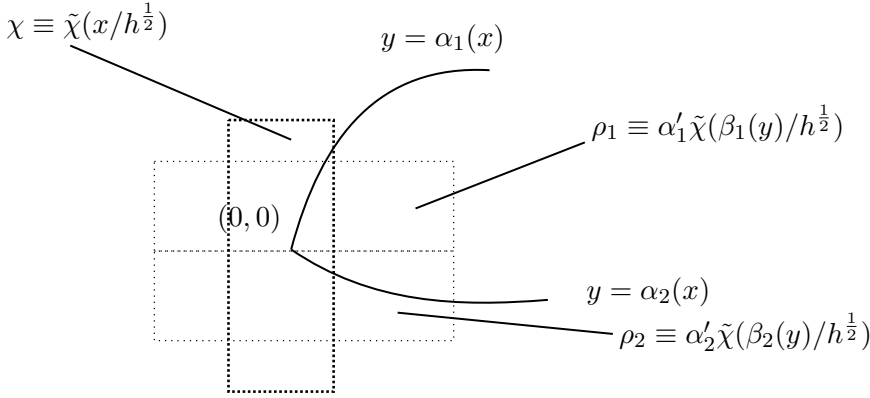}
\caption{\label{F:corner-chi} $\Omega$ in a neighbourhood of a corner
  and the functions $\chi$, $\rho_1$, and $\rho_2$.}
\hfill
\end{figure}

Integrating
by parts in the $x$ direction on the interior terms, we have:
\begin{align*}
 \int_\Omega  ([-h^2 \Delta -1, \chi \p_x ] \phi_h ) \phi_h dV & =  -2
 \int_\Omega  \chi_x (h^2 \p_x^2 \phi_h ) \phi_h dV + \O(1) \\
& = -2 \int_{-r}^0 \int_{x = \beta_2(y)}^r \chi_x (h^2 \p_x^2
  \phi_h ) \phi_h dx dy \\
  & \quad  -2 \int_0^r \int_{x = \beta_1(y)}^r \chi_x (h^2 \p_x^2
  \phi_h ) \phi_h dx dy  + \O(1) \\
  & = : I_1 + I_2 + \O(1).
  \end{align*}
Let us examine $I_1$ first:
\begin{align*}
  I_1 & = 2 \int_{-r}^0 \int_{\beta_2(y)}^r h \chi_{xx} (h
  \p_x \phi_h ) \phi_h dx dy + 2 \int_{-r}^0 \int_{\beta_2(y)}^r  \chi_{x} (|h
  \p_x \phi_h |^2 )dx dy \\
  & \quad -2 \int_{y = -r}^0 h \chi_x (h \p_x \phi_h ) \phi_h
  |_{x = \beta_2(y)}^{x = r} dy \\
  & = 
  2 \int_{-r}^0 \int_{\beta_2(y)}^r  \chi_{x} |h
  \p_x \phi_h |^2 dx dy \\
  & \quad +2 \int_{y = -r}^0 (h \chi_x (h \p_x \phi_h )
  \phi_h)(\beta_2(y), y) dy + \O(1),
\end{align*}
since $\chi$ has support in $x \leq 2 \epsilon \ll r$, and $h
\chi_{xx} = \O(1)$.

Similarly,
\begin{align*}
I_2 & = 
  2 \int_0^r \int_{ \beta_1(y)}^r  \chi_{x} (|h
  \p_x \phi_h |^2 dx dy \\
  & \quad +2 \int_0^r (h \chi_x (h \p_x \phi_h )
  \phi_h)(\beta_1(y), y) dy + \O(1).
\end{align*}
Summing, we have
\begin{align*}
  \int_{\Omega} & ([-h^2 \Delta -1, \chi \p_x ] \phi_h ) \phi_h dV \\
  & = I_1 + I_2 + \O(1)  \\
  & = 
  2  \int_\Omega \chi_{x} |h
  \p_x \phi_h |^2 dx dy 
  \\
  & \quad +2 \int_{y = -r}^0 (h \chi_x (h \p_x \phi_h )
  \phi_h)(\beta_2(y), y) dy \\
  & \quad +2 \int_{y = 0}^r (h \chi_x (h \p_x \phi_h )
  \phi_h)(\beta_1(y), y) dy + \O(1).
  \end{align*}
For the two boundary terms, we change variables $y = \alpha_2(x)$ and $y
= \alpha_1(x)$ respectively to get
\begin{align}
   \int_{\Omega} & ([-h^2 \Delta -1, \chi \p_x ] \phi_h ) \phi_h dV
   \notag \\
& = 
  2  \int_\Omega \chi_{x} (|h
  \p_x \phi_h |^2 dx dy -2  \int_0^r  \alpha_2'(x)(h \chi_x (h \p_x \phi_h )
  \phi_h)(x, \alpha_2(x)) dx \notag \\
  & \quad +2  \int_0^r \alpha_1'(x) (h \chi_x (h \p_x \phi_h )
  \phi_h)(x , \alpha_1(x)) dx + \O(1). \label{E:corner-bdy-x-1}
  \end{align}
Note the sign change on the second integral to correct for reversed
orientation in the $x$ direction.

We now want to employ a  similar argument with $\p_y$.  However, our
function $\rho$ cannot be globally defined if we want to write $\rho$
in terms of $\chi$ on the boundary, since we are not assuming any
relation between $\alpha_1$ and $\alpha_2$.  Let
\[
\Omega_1 = \Omega \cap \{ x \leq r \} \cap \{ y \geq 0 \},
\]
and
\[
\Omega_2 = \Omega \cap \{ x \leq r \} \cap \{ y \leq 0 \}
\]
be the top and bottom parts of $\Omega$ near $(0,0)$.  For $j = 1,2$, let

%

\begin{equation} \label{rhoj}
\rho_j(x,y) := \alpha_j'(xy/\alpha_j(x)) \tchi( \beta_j(y) /h^\half) \tpsi
(x/\epsilon) \tpsi (y/ \epsilon), \quad j=1,2.
\end{equation}\

See Figure \ref{F:corner-chi} for a picture of the setup.  The choice of argument $xy/\alpha_j(x)$ in \eqref{rhoj} is for some cancellation at $y = \alpha_j(x)$ and $y =0$.  The way we have chosen coordinates ensures that $x/\alpha_j(x)$ is smooth and bounded in our domain.

Let us record some facts about the $\rho_j$'s.  First, along $y =
\alpha_j$, $j = 1,2$, we have
\begin{equation}
  \label{E:rho-y-j}
\rho_j ( x , \alpha_j(x)) = \alpha_j'(x) \tchi(x/h^\half)
\tpsi(x/\epsilon) \tpsi(\alpha_j(x)/\epsilon) = \alpha_j' \chi(x,
\alpha_j(x)).
\end{equation}
Along $y = 0$, $\rho_j = 0$, since $\tchi$ is an odd function and
$\beta_j(0) = 0$ for $j = 1,2$.  Along
$y = \alpha_j$,
\begin{align*}
\p_y \rho_j(x, \alpha_j(x)) & = h^{-\half} \alpha_j' (x)
\beta_j'(\alpha_j(x)) \tchi'(x/h^\half)
\tpsi(x/\epsilon)\tpsi(\alpha_j(x)/\epsilon) + A_j + \O(1) \\
& =  h^{-\half} \tchi'(x/h^\half)
\tpsi(x/\epsilon)\tpsi(\alpha_j(x)/\epsilon) + A_j + \O(1) \\
& = \p_x \chi(x, \alpha_j(x)) + A_j + \O(1).
\end{align*}
Here $A_j$ is the term we get when the derivative lands on the $\alpha_j'(xy/\alpha_j(x))$:
\[
A_j 
= (x/\alpha_j(x)) \alpha_j''(xy/\alpha_j(x)) \tchi( \beta_j(y) /h^\half) \tpsi
(x/\epsilon) \tpsi (y/ \epsilon).
\]
Along $y = \alpha_j$, this reduces to
\[
A_j =  (x/\alpha_j(x)) \alpha_j''(x) \tchi( x/h^\half) \tpsi
(x/\epsilon) \tpsi (\alpha_j(x)/ \epsilon).,
\]
and along $y = 0$, $A_j = 0$.

\begin{remark}
\label{R:corner-vanish}
We single out the behaviour of $A_j$ because it is still singular due to the $\tchi(\beta_j/h^\half)$.
Indeed, we have
$
A_j = \tchi( x/h^\half) A_{j1}$
where
\[
\p_x^k A_{j1} = \O (1).
\]
Here the implicit $\O(1)$ errors come from differentiating the $\tpsi$ functions and are supported away from the corner.  We know  $\tchi (\beta_j(y)/h^\half)$ vanishes at $x = y = 0$, and we will use these two observations to integrate by parts along the boundary.  
\end{remark}

Finally, along $y =
0$, we have $\tchi'(0) = 1/2$ and $\tpsi'(0) = 0$, so that
\begin{align}
  \p_y \rho_j(x,0) &= h^{-\half}  \tchi'(0) \tpsi(x/\epsilon) \tpsi(0) + \O(1)\\
  & = (h^{-1/2}/2) \tpsi(x/\epsilon) + \O(1).\label{E:rho-y-0}
  \end{align}


Now consider the vector field $\rho_1 \p_y$ on $\Omega_1$.   The same commutator computation and integrations by
parts in $y$ yields the following:
\begin{align}
  \int_{\Omega_1} & ([-h^2 \Delta -1, \rho_1 \p_y ] \phi_h ) \phi_h
  dV \notag \\
  & = 
  -2 \int_{\Omega_1} \rho_{1,y} (h^2 \p_y^2 \phi_h ) \phi_h dV  + \O(1) \notag \\
  & = -2 \int_{x = 0}^r \int_{y =0 }^{y = \alpha_1(x)}
  \rho_{1,y} (h^2 \p_y^2 \phi_h ) \phi_h dy dx + \O(1)\notag \\
  & = 
  2 \int_{x = 0}^r \int_{y = 0}^{y = \alpha_1(x)}
  \rho_{1,y} |h \p_y \phi_h |^2 dy dx \notag  - 2 \int_{x = 0}^r h \rho_{1,y} (h \p_y \phi_h ) \phi_h
  |_{y = 0}^{y = \alpha_1(x)} dx + \O(1) \notag \\
  & = 
  2 \int_{\Omega_1}
  \rho_{1,y} |h \p_y \phi_h |^2 dy dx  - 2 \int_{x = 0}^r (h \rho_{1,y} (h \p_y \phi_h ) \phi_h)(x,
  \alpha_1(x)) dx \notag
  \\
  & \quad + 2 \int_{x = 0}^r (h \rho_{1,y} (h \p_y \phi_h ) \phi_h)(x,
  0) dx + \O(1) \notag \\
  & = 2 \int_{\Omega_1}
  \rho_{1,y} |h \p_y \phi_h |^2 dy dx \notag \\
  & \quad - 2 \int_{x = 0}^r (h (\chi_x+A_j) (h \p_y \phi_h ) \phi_h)(x,
  \alpha_1(x)) dx \notag 
  \\
  & \quad + 2 \int_{x = 0}^r(h^{-1/2}/2) \tpsi(x/\epsilon) (h  (h \p_y \phi_h ) \phi_h)(x,
  0) dx + \O(1). \label{E:corner-bdy-y-1001} 
  \end{align}
  
  In order to estimate the term with $A_j$ in \eqref{E:corner-bdy-y-1001}, first rewrite the integral in tangent coordinates $\tau$: 
  \begin{align*}
  \int_{x = 0}^r (h A_j (h \p_y \phi_h ) \phi_h)(x,
  \alpha_1(x)) dx  = \int_{\tau = 0}^r h \tA_j(\tau) (h \p_\tau \phi_j) \bar{\phi}_h d \tau 
 \end{align*}
 where $\tA_j$ is $A_j$ written in tangent coordinates together with the arclength factors after changing variables.  The precise form is not important, just that it satisfies the same order of estimates as $A_j$. 
 Continuing, using Remark \ref{R:corner-vanish}:
  \begin{align}
  \int_{\tau = 0}^r  h \tA_j(\tau) (h \p_\tau \phi_h) \bar{\phi}_j d \tau & = \frac{h^2}{2} \int_0^r \tA_j \p_\tau |\phi_h|^2 d \tau \notag \\
  & = -\frac{h^2}{2} \int_0^r (\p_\tau \tA_j) | \phi_h|^2 d \tau \notag \\ 
  & = \O(h^{1/2}) \label{E:Aj-control}
  \end{align}
  from Lemma \ref{L:Sobolev} with $\eta>0$ small but independent of $h$.
  Inserting this estimate into \eqref{E:corner-bdy-y-1001}, we have
  \begin{align}
   \int_{\Omega_1} & ([-h^2 \Delta -1, \rho_1 \p_y ] \phi_h ) \phi_h
  dV \\
& = 2 \int_{\Omega_1}
  \rho_{1,y} |h \p_y \phi_h |^2 dV  - 2 \int_{x = 0}^r (h \chi_x (h \p_y \phi_h ) \phi_h)(x,
  \alpha_1(x)) dx \notag \\
  & \quad + h^\half\int_{x = 0}^r \tpsi(x/\epsilon)   (h \p_y \phi_h ) \phi_h(x,
  0) dx + \O(1)
  . \label{E:corner-bdy-y-1}
\end{align}
%
%

A similar computation on $\Omega_2$ using the vector field $\rho_2
\p_y$ gives
\begin{align}
   \int_{\Omega_2} & ([-h^2 \Delta -1, \rho_2 \p_y ] \phi_h ) \phi_h
  dV \notag \\
  & = 
  -2 \int_{\Omega_2} \rho_{2,y} (h^2 \p_y^2 \phi_h ) \phi_h dV  + \O(1) \notag \\
  & = -2 \int_{x = 0}^r \int_{y =\alpha_2 }^{y = 0}
  \rho_{2,y} (h^2 \p_y^2 \phi_h ) \phi_h dy dx + \O(1)\notag \\
  & = 
  2 \int_{x = 0}^r \int_{y = \alpha_2}^{y = 0}
  \rho_{2,y} |h \p_y \phi_h |^2 dy dx  - 2 \int_{x = 0}^r h \rho_{1,y} (h \p_y \phi_h ) \phi_h
  |_{y = \alpha_2}^{y = 0} dx + \O(1) \notag \\
  & =  2 \int_{\Omega_2}
  \rho_{2,y} |h \p_y \phi_h |^2 dV  - h^\half \int_{x = 0}^r \tpsi(x/\epsilon) (h \p_y \phi_h)
  \phi_h (x,0) dx \notag \\
  & \quad + 2 \int_{x = 0}^r h \chi_x ( h \p_y \phi_h) \phi_h (x
  , \alpha_2(x)) dx + \O(1).
  \label{E:corner-bdy-y-2}
  \end{align}

Summing \eqref{E:corner-bdy-y-1} and \eqref{E:corner-bdy-y-2} and
making the obvious cancellations, we have
\begin{align}
  \int_{\Omega_1} & ([-h^2 \Delta -1, \rho_1 \p_y ] \phi_h) \phi_h dV
  + \int_{\Omega_2}  ([-h^2 \Delta -1, \rho_2 \p_y ] \phi_h) \phi_h
  dV \notag \\
  & = 2 \int_{\Omega_1}
  \rho_{1,y} |h \p_y \phi_h |^2 dV + 2 \int_{\Omega_2}
  \rho_{2,y} |h \p_y \phi_h |^2 dV \notag  \\
  & \quad  - 2 \int_{x = 0}^r (h \chi_x (h \p_y \phi_h ) \phi_h)(x,
  \alpha_1(x)) dx \notag \\
  & \quad + 2 \int_{x = 0}^r h \chi_x ( h \p_y \phi_h) \phi_h (x
  , \alpha_2(x)) dx + \O(1). \label{E:corner-y-both}
  \end{align}

We now use the Neumann boundary conditions on $\phi_h$ and sum
\eqref{E:corner-bdy-x-1} and \eqref{E:corner-y-both}.
On the top
segment where $y \geq 0$, we have \eqref{E:nuxy-1} and
\eqref{E:nuxy-2} so 
\[
0 = \p_\nu \phi_h = -\frac{\alpha_1'}{\kappa_1} \p_x \phi_h +
\frac{1}{\kappa_1} \p_y \phi_h.
\]
Then $\p_y \phi_h = \alpha_1' \p_x \phi_h$ on the upper section.
Similarly, on the bottom section we have \eqref{E:nuxy-3} and
\eqref{E:nuxy-4} so that $\p_y \phi_h =  \alpha_2' \p_x
\phi_h$.  Consequently, \eqref{E:corner-y-both} becomes
\begin{align}
 \int_{\Omega_1} & ([-h^2 \Delta -1, \rho_1 \p_y ] \phi_h) \phi_h dV
  + \int_{\Omega_2}  ([-h^2 \Delta -1, \rho_2 \p_y ] \phi_h) \phi_h
  dV \notag \\
  & = 2 \int_{\Omega_1}
  \rho_{1,y} |h \p_y \phi_h |^2 dV + 2 \int_{\Omega_2}
  \rho_{2,y} |h \p_y \phi_h |^2 dV \notag  \\
  & \quad  - 2 \int_{x = 0}^r (h \chi_x (\alpha_1' h \p_x \phi_h ) \phi_h)(x,
  \alpha_1(x)) dx \notag \\
  & \quad + 2 \int_{x = 0}^r h \chi_x (\alpha_2' h \p_x \phi_h) \phi_h (x
  , \alpha_2(x)) dx + \O(1). \label{E:corner-y-both-2}
  \end{align}
Now summing \eqref{E:corner-bdy-x-1} and \eqref{E:corner-y-both-2} and
making the obvious cancellations, we
have
\begin{align}
   \int_{\Omega} & ([-h^2 \Delta -1, \chi \p_x ] \phi_h ) \phi_h dV
   \notag \\
   &\quad + \int_{\Omega_1}  ([-h^2 \Delta -1, \rho_1 \p_y ] \phi_h) \phi_h dV
  + \int_{\Omega_2}  ([-h^2 \Delta -1, \rho_2 \p_y ] \phi_h) \phi_h
  dV \notag \\
  & = 2  \int_\Omega \chi_{x} (|h
  \p_x \phi_h |^2 dV + 
  2 \int_{\Omega_1}
  \rho_{1,y} |h \p_y \phi_h |^2 dV + 2 \int_{\Omega_2}
  \rho_{2,y} |h \p_y \phi_h |^2 dV  + \O(1).  \label{E:comm-master-1}
  \end{align}

It remains to compute the commutators  on the LHS of (\ref{E:comm-master-1}). By Green's formula,
\[
\int_\Omega ([-h^2 \Delta -1, \chi(x,y) \p_x ] \phi_h) \phi_h dV = -
\int_{\p \Omega} (h \p_\nu \chi h \p_x \phi_h ) \phi_h dS
\]
and for $j = 1,2$
\begin{align*}
\int_{\Omega_j} &([-h^2 \Delta -1, \rho_j(x,y) \p_y ] \phi_h) \phi_h
dV \\
& = -
\int_{\p \Omega_j} (h \p_\nu \rho_j h \p_y \phi_h ) \phi_h dS +
\int_{\p \Omega_j} (\rho_j h \p_y \phi_h )(h \p_\nu \phi_h) dS \\
& = -\int_{\p \Omega_j} (h \p_\nu \rho_j h \p_y \phi_h ) \phi_h dS.
\end{align*}

Here the second integral in the second line is zero since $\phi_h$ has
Neumann boundary conditions along the boundary $y = \alpha_1$, and
$\rho_j = 0$ along the line $y = 0$.

On the upper segment, we use that
$
\p_\nu = -\frac{\alpha_1'}{\kappa_1} \p_x + \frac{1}{\kappa_1} \p_y , $ and 
$
\p_x = \frac{1}{\kappa_1} \p_\tau - \frac{\alpha_1'}{\kappa_1} \p_\nu$ to get
\begin{align*}
  h \p_\nu \chi h \p_x \phi_h & =  \chi h \p_x h \p_\nu \phi_h + [h
    \p_\nu , \chi h \p_x ] \phi_h \\
  & = -\frac{\alpha_1'}{\kappa_1}  \chi h^2
\p_\nu^2 \phi_h -
   \frac{\alpha_1'}{\kappa_1^2} h \chi_x
 h \p_\tau \phi_h + \O(h) h \p_\tau \phi_h .
\end{align*}
  
Similarly, on the lower segment, $
\p_\nu = \frac{\alpha_2'}{\kappa_2} \p_x - \frac{1}{\kappa_2} \p_y $ and 
$
\p_x = \frac{1}{\kappa_2} \p_\tau + \frac{\alpha_2'}{\kappa_2} \p_\nu $ so
that 
\begin{align*}
  h \p_\nu \chi h \p_x \phi_h & =\frac{\alpha_2'}{\kappa_2}  \chi h^2
\p_\nu^2 \phi_h + \frac{\alpha_2'}{\kappa_2^2}h \chi_x h \p_\tau \phi_h +
\O(h) h \p_\tau
\phi_h.
\end{align*}

Plugging in, we have
\begin{align}
  \int_\Omega & ([-h^2 \Delta -1, \chi(x,y) \p_x ] \phi_h) \phi_h dV
 \notag  \\
  &  =-
  \int_{\p \Omega} (h \p_\nu \chi h \p_x \phi_h ) \phi_h dS \notag \\
  & = - \int_{\p \Omega \cap \{ y \geq 0 \} } (-\frac{\alpha_1'}{\kappa_1}  \chi h^2
\p_\nu^2 \phi_h -
   \frac{\alpha_1'}{\kappa_1^2} h \chi_x
   h \p_\tau \phi_h + \O(h) h \p_\tau \phi_h) \phi_h dS \notag \\
   & \quad - \int_{\p \Omega \cap \{ y \leq 0 \} } ( \frac{\alpha_2'}{\kappa_2}  \chi h^2
\p_\nu^2 \phi_h + \frac{\alpha_2'}{\kappa_2^2} h \chi_xh \p_\tau \phi_h + \O(h)
h \p_\tau \phi_h) \phi_h dS \notag \\
& = - \int_{\p \Omega \cap \{ y \geq 0 \} } (-\frac{\alpha_1'}{\kappa_1}  \chi h^2
\p_\nu^2 \phi_h -
   \frac{\alpha_1'}{\kappa_1^2} h \chi_x
   h \p_\tau \phi_h ) \phi_h dS \notag \\
   & \quad - \int_{\p \Omega \cap \{ y \leq 0 \} } ( \frac{\alpha_2'}{\kappa_2}  \chi h^2
\p_\nu^2 \phi_h + \frac{\alpha_2'}{\kappa_2^2} h \chi_x h \p_\tau \phi_h )
\phi_h dS + \O(1), \label{E:chi-IBP-1}
\end{align}
where we have again used integration by parts along the boundary and Sobolev embedding on the implicit $\O(h)$
boundary terms supported away from the corner, just as we did in \eqref{E:bdy-h-Sobolev-1}-\eqref{E:bdy-h-Sobolev-2}.

For the computations involving the vector fields $\rho_j \p_y$, we
have  by Green's formula 
\begin{align*}
  \int_{\Omega_j}&  ([-h^2 \Delta -1, \rho_j \p_y ] \phi_h ) \phi_h dV
  \\
  & = - \int_{\p \Omega_j} (h \p_\nu \rho_jh \p_y \phi_h) \phi_h dS
  \\
  & = - \int_{\{ y = \alpha_1(x)\}} (h \p_\nu \rho_jh \p_y \phi_h)
  \phi_h dS - \int_{\{ y = 0 \}} (h \p_\nu \rho_jh \p_y \phi_h)
  \phi_h dS,
\end{align*}
since $\tpsi(x /\epsilon)$ has compact support in $\{ x \leq 2
\epsilon \ll r \}$.
Using the same computations  which led to \eqref{E:chi-IBP-1}, on $\{
y = \alpha_1 \}$, we have
\[
h \p_\nu \rho_1 h \p_y \phi_h = \frac{1}{\kappa_1} \rho_1 h^2 \p_\nu^2
\phi_h + \frac{\alpha_1'}{\kappa_1^2} h\rho_{1,y} h \p_\tau \phi_h +
\O(h) h \p_\tau \phi_h.
\]
On $\{ y = 0 \}$, from
$\Omega_1$, we have $\p_\nu = - \p_y$, so that
\begin{align*}
  h \p_\nu \rho_1 h \p_y \phi_h & = -h \p_y \rho_1 h \p_y \phi_h \\
  & = -h \rho_{1,y} h \p_y \phi_h - \rho_1 h^2 \p_y^2 \phi_h \\
  & = -h \rho_{1,y} h \p_y \phi_h \\
  & = -(h^\half/2) \tpsi(x/\epsilon) h \p_y \phi_h,
\end{align*}
since $\rho_1(x,0) = 0$.  In the last line we have used \eqref{E:rho-y-0}.
Putting this together, we have
\begin{align}
  \int_{\Omega_1}&  ([-h^2 \Delta -1, \rho_1 \p_y ] \phi_h ) \phi_h dV
 \notag  \\ 
  & = - \int_{\p \Omega_1} (h \p_\nu \rho_1h \p_y \phi_h) \phi_h dS
  \notag \\
  & = -\int_{\{ y = \alpha_1 \}} (h \p_\nu \rho_1 h \p_y \phi_h)
  \phi_h dS  - \int_{\{y = 0 \}} (h \p_\nu \rho_1 h \p_y \phi_h)
  \phi_h dS \notag  \\
  & = -\int_{\{ y = \alpha_1\}} ( \frac{1}{\kappa_1} \rho_1 h^2 \p_\nu^2
\phi_h + \frac{\alpha_1'}{\kappa_1^2} h\rho_{1,y} h \p_\tau \phi_h +
\O(h) h \p_\tau \phi_h) \phi_h dS \notag \\
& \quad - \int_{\{ y = 0 \}} (-h \rho_{1,y} h \p_y \phi_h) \phi_h dS
\notag \\
& = -\int_{\{ y = \alpha_1\}} ( \frac{1}{\kappa_1} \rho_1 h^2 \p_\nu^2
\phi_h + \frac{\alpha_1'}{\kappa_1^2} h\rho_{1,y} h \p_\tau \phi_h)
\phi_h dS \notag \\
& \quad - \int_{\{ y = 0 \}} (-(h^\half/2) \tpsi(x/\epsilon) h \p_y \phi_h) \phi_h dS
+ \O(1).\label{E:comm-om-1}
\end{align}
Here we have once again used the Sobolev embedding on the implicit
$\O(h)$ boundary terms just as we did in \eqref{E:bdy-h-Sobolev-1}-\eqref{E:bdy-h-Sobolev-2}.

In a similar fashion, we compute for $\Omega_2$:
\begin{align}
  \int_{\Omega_2} & ([-h^2 \Delta -1, \rho_2 \p_y ] \phi_h ) \phi_h
  dV \notag 
  \\
  & = - \int_{\p \Omega_2} (h \p_\nu \rho_2 h \p_y \phi_h) \phi_h dS
  \notag 
  \\
  & =  -\int_{\{ y = \alpha_2\}} ( -\frac{1}{\kappa_2} \rho_2 h^2 \p_\nu^2
\phi_h - \frac{\alpha_2'}{\kappa_1^2} h\rho_{2,y} h \p_\tau \phi_h)
\phi_h dS \notag \\
& \quad - \int_{\{ y = 0 \}} ((h^\half/2) \tpsi(x/\epsilon) h \p_y \phi_h) \phi_h dS
+ \O(1). \label{E:comm-om-2}
\end{align}
Here in the last line we have used that, from $\Omega_2$, $\p_\nu =
\p_y$ along $\{ y = 0 \}$.

Summing \eqref{E:comm-om-1} and \eqref{E:comm-om-2}, we have
\begin{align}
   \int_{\Omega_1} & ([-h^2 \Delta -1, \rho_1 \p_y ] \phi_h ) \phi_h
   dV
   +  \int_{\Omega_2}  ([-h^2 \Delta -1, \rho_2 \p_y ] \phi_h ) \phi_h
   dV\notag  \\
   & = 
-\int_{\{ y = \alpha_1\}} ( \frac{1}{\kappa_1} \rho_1 h^2 \p_\nu^2
\phi_h + \frac{\alpha_1'}{\kappa_1^2}h \rho_{1,y} h \p_\tau \phi_h)
\phi_h dS \notag \\
& \quad + \int_{\{ y = 0 \}} ((h^\half/2) \tpsi(x/\epsilon) h \p_y
\phi_h) \phi_h dS \notag \\
& \quad
 +\int_{\{ y = \alpha_2\}} ( \frac{1}{\kappa_2} \rho_2 h^2 \p_\nu^2
\phi_h + \frac{\alpha_2'}{\kappa_2^2}h \rho_{2,y} h \p_\tau \phi_h)
\phi_h dS \notag \\
& \quad - \int_{\{ y = 0 \}} ((h^\half/2) \tpsi(x/\epsilon) h \p_y \phi_h) \phi_h dS
+ \O(1) \notag \\
& = -\int_{\{ y = \alpha_1\}} ( \frac{1}{\kappa_1} \rho_1 h^2 \p_\nu^2
\phi_h + \frac{\alpha_1'}{\kappa_1^2} h\rho_{1,y} h \p_\tau \phi_h)
\phi_h dS \notag \\
& \quad 
 +\int_{\{ y = \alpha_2\}} ( \frac{1}{\kappa_2} \rho_2 h^2 \p_\nu^2
\phi_h + \frac{\alpha_2'}{\kappa_2^2} h\rho_{2,y} h \p_\tau \phi_h)
\phi_h dS 
+ \O(1). \label{E:comm-om-both}
\end{align}

From 
\eqref{E:rho-y-j}, we know that $$\rho_j(x, \alpha_j(x)) =\alpha_j'
\chi(x, \alpha_j(x))$$ and $$\rho_{j,y} (x, \alpha_j(x))= \chi_x( x ,
\alpha_j(x)) + A_j + \O(1),$$ so that
\eqref{E:comm-om-both} becomes
\begin{align}
     \int_{\Omega_1} & ([-h^2 \Delta -1, \rho_1 \p_y ] \phi_h ) \phi_h
   dV
   +  \int_{\Omega_2}  ([-h^2 \Delta -1, \rho_2 \p_y ] \phi_h ) \phi_h
   dV\notag  \\
&
= -\int_{\{ y = \alpha_1\}} ( \frac{1}{\kappa_1} \alpha_1' \chi h^2 \p_\nu^2
\phi_h + \frac{\alpha_1'}{\kappa_1^2} h(\chi_x + A_1) h \p_\tau \phi_h)
\phi_h dS \notag \\
& \quad 
 +\int_{\{ y = \alpha_2\}} ( \frac{1}{\kappa_2} \alpha_2' \chi  h^2 \p_\nu^2
\phi_h + \frac{\alpha_2'}{\kappa_2^2} h(\chi_x + A_2) h \p_\tau \phi_h)
\phi_h dS 
+ \O(1) \\
& =  -\int_{\{ y = \alpha_1\}} ( \frac{1}{\kappa_1} \alpha_1' \chi h^2 \p_\nu^2
\phi_h + \frac{\alpha_1'}{\kappa_1^2} h\chi_x  h \p_\tau \phi_h)
\phi_h dS \notag \\
& \quad 
 +\int_{\{ y = \alpha_2\}} ( \frac{1}{\kappa_2} \alpha_2' \chi  h^2 \p_\nu^2
\phi_h + \frac{\alpha_2'}{\kappa_2^2} h\chi_x h \p_\tau \phi_h)
\phi_h dS 
+ \O(1) ,
\label{E:comm-om-both-2}
\end{align}
where we have used an argument similar to \eqref{E:Aj-control} to control the boundary terms with the $A_j$s.

Summing \eqref{E:chi-IBP-1} and \eqref{E:comm-om-both-2}, we have
\begin{align}
  \int_\Omega & ([-h^2 \Delta -1, \chi \p_x] \phi_h) \phi_h dV \notag
  \\
  & \quad + 
      \int_{\Omega_1}  ([-h^2 \Delta -1, \rho_1 \p_y ] \phi_h ) \phi_h
   dV
   +  \int_{\Omega_2}  ([-h^2 \Delta -1, \rho_2 \p_y ] \phi_h ) \phi_h
   dV\notag \\
   & = 
 - \int_{\p \Omega \cap \{ y \geq 0 \} } (-\frac{\alpha_1'}{\kappa_1}  \chi h^2
\p_\nu \phi_h -
   \frac{\alpha_1'}{\kappa_1^2} h \chi_x
   h \p_\tau \phi_h ) \phi_h dS \notag \\
   & \quad - \int_{\p \Omega \cap \{ y \leq 0 \} } ( \frac{\alpha_2'}{\kappa_2}  \chi h^2
\p_\nu \phi_h + \frac{\alpha_2'}{\kappa_2^2} h \chi_xh \p_\tau \phi_h )
\phi_h dS \notag \\
& \quad -\int_{\{ y = \alpha_1\}} ( \frac{1}{\kappa_1} \alpha_1' \chi h^2 \p_\nu^2
\phi_h + \frac{\alpha_1'}{\kappa_1^2} h\chi_x h \p_\tau \phi_h)
\phi_h dS \notag \\
& \quad 
 +\int_{\{ y = \alpha_2\}} ( \frac{1}{\kappa_2} \alpha_2' \chi  h^2 \p_\nu^2
\phi_h + \frac{\alpha_2'}{\kappa_2^2} h\chi_x h \p_\tau \phi_h)
\phi_h dS 
+ \O(1). \label{E:master-comm-2}
\end{align}
All of the displayed boundary terms in \eqref{E:master-comm-2} cancel,
so that
\begin{align}
 \int_\Omega & ([-h^2 \Delta -1, \chi \p_x] \phi_h) \phi_h dV \notag
  \\
  & \quad + 
      \int_{\Omega_1}  ([-h^2 \Delta -1, \rho_1 \p_y ] \phi_h ) \phi_h
   dV
   +  \int_{\Omega_2}  ([-h^2 \Delta -1, \rho_2 \p_y ] \phi_h ) \phi_h
   dV\notag \\
   & = \O(1).
 \label{E:master-comm-21}
\end{align}

Finally, equating 
\eqref{E:comm-master-1} and \eqref{E:master-comm-2}, we have shown
\[
2  \int_\Omega \chi_{x} (|h
  \p_x \phi_h |^2 dV + 
  2 \int_{\Omega_1}
  \rho_{1,y} |h \p_y \phi_h |^2 dV + 2 \int_{\Omega_2}
  \rho_{2,y} |h \p_y \phi_h |^2 dV
= 
\O(1).
\]

Using the same estimates as in \eqref{E:chixsupp}, we have that
\[
\chi_x  \geq  h^{-1/2} \gamma(x/h^{1/2})\gamma(y/h^{1/2})
\gamma(y/h^{1/2}) - \O(1) \]
and on each of $\Omega_j$,
\[
\rho_{j,y} \geq  c_0 h^{-1/2} \gamma(x/h^{1/2})\gamma(y/h^{1/2}),
\]
for some $c_0>0$ independent of $h$, 
so arguing as in \eqref{E:comm-sum-1}, we finally get
\[
\int_{B((0,0), h^\half)} h^{-\half} | \phi_h |^2 dV = \O(1).
\]

{\bf Step 2: $\delta = 2/3$.}

We are now ready to bootstrap the estimate for $\delta = 2/3$.  The
argument proceeds exactly as in the $\delta = 1/2$ case, but now some
of the error terms are no longer so easy to absorb.  The bootstrap from $\delta = 1/2$ to $\delta = 2/3$ forms the basis for our induction argument.  The idea is to use a Sobolev-type estimate and the estimate in a ball of radius $\sim h^{1/2}$ to control one of the largest boundary terms which show up when our cutoff function is on scale $h^{2/3}$.  The other largest boundary terms will cancel similar to the $h^{1/2}$ case.

 We begin with the case where $p_0$ is not a corner, starting with
 defining the cutoff $\chi$ as in 
 \eqref{E:chidef}.  For $\epsilon>0$ small but independent of $h$, let

 \begin{equation}
  \chi(x,y) = \tchi ( x/h^{2/3}) \tpsi (x/\epsilon) \tpsi(
  y/\epsilon)\label{E:chidef-101}.
 \end{equation}\
 We similarly define
 \[
 \rho(x,y) := \alpha'(x) \tchi ( \beta(y)/ h^{2/3}) \tpsi (x/\epsilon) \tpsi(
y/\epsilon). 
\]
 
Observe the only difference  in (\ref{E:chidef-101}) versus the cutoff in (\ref{E:chidef})  is the $h^{-2/3}$ appearing instead of $h^{-1/2}$.
This is good, since we will once again need some boundary terms to
cancel.  The argument is identical to the argument in the $\delta =
1/2$ case except for one piece: $\chi_{xx}$ is no longer $\O(h^{-1})$
but instead is $\O(h^{-4/3})$.  We will have to work harder to control
this.

Beginning with the commutator, 
since $\chi_y$ and $\chi_{yy}$ are both 
bounded, we have
\begin{align*}
  \int_{\Omega} & ([-h^2 \Delta -1, \chi \p_x ] \phi_h ) \phi_h dV \\
  & = \int_{\Omega} ((-2 \chi_x h^2 \p_x^2 - h \chi_{xx} h \p_x
  - 2 \chi_y h \p_x h \p_y - h\chi_{yy} h \p_x) \phi_h ) \phi_h dV \\
  & = -2 \int_\Omega (\chi_x h^2 \p_x^2 \phi_h ) \phi_h dV
-\int_{\Omega}  h \chi_{xx}( h \p_x \phi_h )
 \phi_h dV
  + \O(1).
  \end{align*}

Let
\[
I = \int_\Omega  h \chi_{xx}( h \p_x \phi_h )
\phi_h dV.
\]
Even though $\chi_{xx} = \O(h^{-4/3})$, we will nevertheless show $I$
is bounded.  Write
\[
I = \int_{-r}^r \int_{\beta(y)}^r h \chi_{xx} (h \p_x
\phi_h) \phi_h dx dy
\]
and integrate by parts:
\begin{align}
  I & = - \int_{-r}^r \int_{\beta(y)}^r 
  (\phi_h) h \p_x (h \chi_{xx} \phi_h) dxdy + \int_{-r}^r h^2
  \chi_{xx} | \phi_h|^2 |_{\beta(y)} ^r dy\notag  \\
  & = - {I} - h^2 \int_{-r}^r \int_{\beta(y)}^r 
  \chi_{xxx} |\phi_h|^2  dxdy  + \int_{-r}^r h^2
  \chi_{xx} | \phi_h|^2 |_{\beta(y)}^r dy. \label{E:I-101}
\end{align}
Let
\[
I_1 = h^2 \int_{-r}^r \int_{\beta(y)}^r 
\chi_{xxx} |\phi_h|^2  dxdy.
\]
We have $\chi_{xxx} = h^{-2}$, so $I_1 = \O(1)$.  We pause briefly
here to observe that the function $\chi_{xxx}$ still has large support
in the $y$ direction, so we cannot use the $\delta = 1/2$
non-concentration estimate here.  We use that for the next term: let
\[
I_2 = \int_{-r}^r h^2
\chi_{xx} | \phi_h|^2 |_{\beta(y)} ^r dy.
\]
As before, the support properties of $\chi$ and its derivatives tells
us
\[
I_2 = - \int_{-r}^r h^2
\chi_{xx} | \phi_h|^2 (\beta(y), y)  dy.
\]


 \begin{remark}  Away from corners, the bound $I_2 = O(1)$ follows from the universal eigenfunction boundary restriction upper bound in \cite{Ta}. Indeed, since $h^2 \chi_{xx} = O(h^{2/3}) \tilde{\chi}_{xx}$,  
$$ I_{2} = O(h^{2/3}) \int_{\partial \Omega} \tilde{\chi}_{xx}| \phi_h |^2  dS = O(1)$$
where the last estimate follows from the Tataru bound $ \int_{\partial \Omega} \tilde{\chi}_{xx} | \phi_h |^2 dV = O(h^{-2/3})$ since $\tilde{\chi}_{xx}$ is supported away from corners. However, since we will need our estimates to hold near corners as well, we give a more direct argument here to bound $I_2.$ \end{remark}

Note that $I_2$ \,  is a boundary integral with support in three different regions in
the $x$ direction.  We have $\chi_{xx} = \O(h^{-4/3})$ for $-3 h^{2/3}
\leq x \leq 3 h^{2/3}$, and $\chi_{xx} = \O(1)$ for $| x | \geq 3
h^{2/3}$.  In the latter region, the boundary integral then has $h^2$,
so Sobolev embedding gives $\O(h)$.  It is on the region $-3 h^{2/3}
\leq x \leq 3h^{2/3}$ where we may encounter a problem. 
Let $[a(h), b(h)]$ be the image in $y$ of $[-3 h^{2/3} , 3 h^{2/3}]$.
Using Lemma \ref{L:Sobolev} with $\eta \sim h^{1/2}$, we have 
\begin{align}
I_2 & = h^{2/3}\int_{[a(h), b(h)]} (h^{4/3} \chi_{xx}) | \phi_h |^2 dS  + \O(1) \notag \\
& = \O(h^{-1/3}) \int_{B(p_0, M h^{2/3})} | \phi_h  |^2 dV + \O(1) \notag \\
& = \O(h^{-1/3}) \int_{B(p_0, h^\half)} | \phi_h|2 dV + \O(1) \notag \\
& = \O(h^{1/6}) + \O(1).\label{sobolev}
\end{align}
  Here $M>0$ is a constant large enough so that
  \[
  \{ (\beta(y), y) : a(h) \leq y \leq b(h) \} \subset B(p_0, M
  h^{2/3}).
  \]
%
  Combining this with the estimate on $I_1$ and plugging into \eqref{E:I-101}, we have
  \[
 2  I = \O(1).
 \]

 Now the computations \eqref{E:boundaryI1}-\eqref{E:comm-sum-1} are identical, including the
 boundary cancellations, leading to
 \begin{align}
    \int_\Omega & ([-h^2 \Delta -1, \chi \p_x ] \phi_h) \phi_h dV 
  \notag \\
  & \quad + \int_\Omega  ([-h^2 \Delta -1, \rho \p_y ] \phi_h) \phi_h
  dV \notag  \\
  & \geq c_0
   \int_\Omega  h^{-2/3} \gamma(x/h^{2/3})
  \gamma(y/h^{2/3}) | \phi_h|^2 
  dV - \O(1) \label{E:comm-sum-101}
  \end{align}
for some $c_0>0$ independent of $h$.

On the other hand,  expanding the commutator, using the Neumann
boundary conditions, and applying Sobolev embedding as in \eqref{E:boundary-102} yields the exact same identity:
  \begin{align*}
       \int_\Omega  & ([-h^2 \Delta -1, \chi \p_x ] \phi_h) \phi_h dV
       \\
       & = \int_{\p \Omega} \left(\chi \frac{
         \alpha'}{\kappa} h^2 \p_\nu^2 \phi_h \right) \phi_h dS \\
       & \quad + \int_{\p \Omega} \left( \frac{ \alpha'}{\kappa^2} h
       \chi_xh \p_\tau \phi_h \right) \phi_h dS + \O(1).
  \end{align*}
  And again, similar computations give
   \begin{align*}
    \int_\Omega & ([-h^2 \Delta -1, \rho \p_y ] \phi_h ) \phi_h dV \\
    & = - \int_{\p \Omega} \left( \rho \frac{1}{\kappa} h^2 \p_\nu^2
    \phi_h \right) \phi_h dS \\
    & \quad - \int_{\p \Omega} \left( \frac{\alpha'}{\kappa^2} h\rho_y h
    \p_\tau \phi_h \right)\phi_h dS  + \O(1).
   \end{align*}
   Again using the same cancellation on the boundary terms,
   we finally arrive at
   \[
   \int_\Omega   ([-h^2 \Delta -1, \chi \p_x ] \phi_h) \phi_h dV
   + 
 \int_\Omega  ([-h^2 \Delta -1, \rho \p_y ] \phi_h ) \phi_h dV =
 \O(1).
 \]
 Comparing to \eqref{E:comm-sum-101}, we have
 \[
   \int_\Omega  h^{-2/3} \gamma(x/h^{2/3})
  \gamma(y/h^{2/3}) | \phi_h|^2 
  dV = \O(1),
  \]
  which completes the proof in the case $p_0$ is not a corner.

To prove the result for $\delta = 2/3$ in the corner case, we again have to be careful with the additional $A_j$s that show up.  Copying the computations in the $\delta = 1/2$ case, we are led to consider the integrals similar to \eqref{E:corner-bdy-y-1001}, but on scale $h^{2/3}$:
\begin{align}
  \int_{\Omega_1} & ([-h^2 \Delta -1, \rho_1 \p_y ] \phi_h ) \phi_h
  dV \notag \\
  & = 2 \int_{\Omega_1}
  \rho_{1,y} |h \p_y \phi_h |^2 dy dx \notag \\
  & \quad - 2 \int_{x = 0}^r (h (\chi_x+A_1) (h \p_y \phi_h ) \phi_h)(x,
  \alpha_1(x)) dx \notag 
  \\
  & \quad + 2 \int_{x = 0}^r(h^{-2/3}/2) \tpsi(x/\epsilon) (h  (h \p_y \phi_h ) \phi_h)(x,
  0) dx + \O(1). \label{E:corner-bdy-y-1002} 
  \end{align}
  
  In order to estimate the term with $A_1$ in \eqref{E:corner-bdy-y-1002} we switch to tangent coordinates $\tau$ just like for \eqref{E:corner-bdy-y-1001}.  But now we have 
  \begin{align*}
  \p_\tau \tA_1 & =  (\p_\tau \tchi( \beta_1(y) /h^{2/3}) )
  (x/\alpha_1(x)) \alpha_1''(xy/\alpha_j(x))\tpsi
(x/\epsilon) \tpsi (y/ \epsilon)
\\
& \quad +  \tchi( \beta_1(y) /h^{2/3}) 
 \p_\tau ( (x/\alpha_1(x)) \alpha_1''(xy/\alpha_1(x))\tpsi
(x/\epsilon) \tpsi (y/ \epsilon)) \\
& = : \tA_{11} + \tA_{12}.
\end{align*}
This means that   
  $\tA_{11} = \O(h^{-2/3})$ with support on $|\tau| \lesssim h^{2/3}$ and $\tA_{12} = \O(1)$.  
  
  Replacing $\tA_1$ with $(\kappa_1/\alpha_1')\tA_1 $ to account for the coordinate change $y \mapsto \tau$ and continuing in tangent coordinates:
  \begin{align*}
  \int_{x = 0}^r (h A_1(h \p_y \phi_h ) \phi_h)(x,
  \alpha_1(x)) dx  & = \int_{\tau = 0}^r h \tA_1(\tau) (h \p_\tau \phi_h) {\phi}_h d \tau \\
  & = \frac{h^2}{2} \int_{\tau = 0}^r \tA_1 \p_\tau | \phi_h|^2 d \tau \\
  & = - \frac{h^2}{2} \int_{\tau = 0}^r (\p_\tau \tA_1) | \phi_h|^2 d \tau \\
  & = - \frac{h^2}{2} \int_{\tau = 0}^r (\tA_{11} + \tA_{12}) | \phi_h|^2 d \tau \\
  & = - \frac{h^2}{2} \int_{\tau = 0}^r \tA_{11} | \phi_h|^2 d \tau + \O(1).
 \end{align*}
Now an application of Lemma \ref{L:Sobolev} with $\eta \lesssim h^{2/3}$ yields
\begin{align*}
\frac{h^2}{2} & \int_{\tau = 0}^r A_{11} | \phi_h|^2 d \tau \\
& = \O(h^{4/3}) \int_{\tau = 0}^r (h^{2/3} A_{11}) | \phi_h|^2 d \tau \\
& = \O(h^{1/3}) \int_{B(p_0, M h^{2/3})} | \phi_h|^2 d V \\
& \leq C h^{1/3} \int_{B(p_0, h^{1/2})} | \phi_h|^2 dV \\
& = \O(h^{5/6}).
\end{align*}

The rest of the proof of the $\delta = 2/3$ case for corners is exactly the same as for $\delta = 1/2$.

%
%
%
%
%
%

  {\bf Step 3 (induction):} $2/3 < \delta <1.$

Our goal now is to prove that for any integer $k >0$, the theorem is
true for $\delta = 1-1/3k$.  The case $k = 1$ has already been shown,
so we are ready for the induction step.

We will need better control over some of the boundary terms than we
have had previously.  We will employ more or less the same cutoffs, so
the same important cancellation will occur, but it is the ``lower
order'' terms we need to estimate.  The issue is that lower order for
the induction means we use the estimates for $\delta = 1-1/3k$ to prove the
estimates for $\delta = 1-1/3(k+1)$.  Since in these cases $\delta
>2/3$, this is more complicated.

In order to fix the ideas and notations, let $\tchi$ and $\tpsi$ be as
in Subsection \ref{SS:preliminaries}.  We work initially away from a corner, but
the proof in the corner case follows line by line as the proof in the
$\delta = 1/2$ case, with one notable exception which we shall point
out as we proceed.

Fix $p_0 \in \p \Omega$ away from a corner and rotate and translate as
above so that $p_0 = (0,0)$, and locally $\p \Omega$ is a graph $y =
\alpha(x)$, $\alpha'(0) \neq 0$.  We also write $\beta = \alpha^{-1}$
so that the boundary can also be written $x = \beta(y)$.  
Let $r>0$ be as in the beginning of the proof, a number
independent of $h$  such that $B(p_0, r)$ does not meet any
corners.  Again, this is just to avoid messy numerology when writing
down our integral formulae.

Fix an integer $k>0$ and let
\[
\eta_k = 1-\frac{1}{3k}
\]
be the corresponding index.  Let 
\begin{equation}
  \label{E:chi-def-k}
  \chi = \tchi(x/h^{\eta_{k+1}})\tpsi^2(x/h^{\eta_{k}})
  \tpsi^2(y/h^{\eta_k}).
\end{equation}
We observe that this cutoff has derivative
$\sim h^{-\eta_{k+1}}$ for $x$ in an $h^{\eta_{k+1}}$ neighbourhood,
but is supported in a neighbourhood of size $h^{\eta_k}$.  In
particular, we record the following facts:
\begin{itemize}
  \item $\chi(x,y) = x/2h^{\eta_{k+1}}$ for $-h^{\eta_{k+1}} \leq x
    \leq h^{\eta_{k+1}}$ and $-h^{\eta_k} \leq y \leq h^{\eta_k}$.

  \item
    $\chi$ is supported in $[-2h^{\eta_k} , 2 h^{\eta_k}]^2$.

  \item
    The support of $\chi_x$ has three connected components in $x$:
    \[
    \chi_x = 1/2h^{\eta_{k+1}}, \,\, | x | \leq h^{\eta_{k+1}},
    \]
    and
    \[
    \chi_x = \O(h^{-\eta_{k+1}}), \,\, | x | \leq 3 h^{\eta_{k+1}};
    \]
    \[
    \chi_x = 0, \,\, 3 h^{\eta_{k+1}} \leq | x | \leq h^{\eta_k};
    \]
    and
    \[
    \chi_x = \O(h^{-\eta_k}), \,\, h^{\eta_k} \leq | x | \leq 2 h^{\eta_k}.
    \]
\end{itemize}
The purpose for replacing $\tpsi$ with $\tpsi^2$ will become apparent
shortly.

{\bf Claim:}  For $h>0$ sufficiently small, we have the estimate
\begin{equation}
  \label{E:chi-k-deriv}
\int_\Omega \chi( | h \p_x \phi_h|^2 + | h \p_y\phi_h|^2)dV = \O(h^{\eta_k}).
\end{equation}
To prove the claim, we will integrate by parts.  We first get rid of
the $\tchi$ part:
\[
| \chi | \leq \tpsi^2(x/h^{\eta_k}) \tpsi^2(y/h^{\eta_k}).
\]
In order to ease notation, let $\psi_k(x) = \tpsi(x/h^{\eta_k})$ and
similarly for $\psi_k(y)$.  
Then we integrate by parts.  Letting $I$ denote the integral (after
removing the $\tchi$):
\begin{align*}
  I & = \int_\Omega  \psi_k^2(x) \psi_k^2(y)( | h \p_x \phi_h|^2 + | h
  \p_y \phi_h|^2)dV \\
  & = \int_\Omega \psi_k^2(x) \psi_k^2(y) (-h^2 \Delta \phi_h) \phi_h dV \\
  & \quad - \int_{\Omega} 2
  h^{1-\eta_k} \tpsi'(x/h^{\eta_k})
  \psi_k(x)\psi_k^2(y) (h \p_x \phi_h) \phi_h dV \\
  & \quad -  \int_{\Omega} 2
  h^{1-\eta_k} \psi_k^2(x/h)
  \tpsi'(y/h^{\eta_k})\psi_k(y) (h \p_y \phi_h) \phi_h dV \\
  & \quad + \int_{\p \Omega} h \psi_k^2(x)
  \psi_k^2(y) ( h \p_\nu \phi_h) \phi_h dS.
\end{align*}
The last term is zero due to the Neumann boundary conditions.  For the
remaining terms, observe that $1 - \eta_k >0$ so we can estimate the
second and third terms using Cauchy's inequality:
\begin{align*}
\Bigg| \int_{\Omega} & 2
  h^{1-\eta_k} \tpsi'(x/h^{\eta_k})
  \psi_k(x)\psi_k^2(y) (h \p_x \phi_h) \phi_h dV \\
  & \quad +  \int_{\Omega} 2
  h^{1-\eta_k} \psi_k^2(x)
  \tpsi'(y/h^{\eta_k})\psi_k(y) (h \p_y \phi_h) \phi_h dV \Bigg|
  \\
  & \leq C h^{1 - \eta_k} \int_{[-2 h^{\eta_k}, 2 h^{\eta_k}]^2}
  \psi_k^2(y)  (\psi_k^2(x)| h \p_x \phi_h|^2 +
(\psi_k'(x))^2  |\phi_h|^2) dV \\
  & \quad + C h^{1 - \eta_k} \int_{[-2 h^{\eta_k}, 2 h^{\eta_k}]^2}
  \psi_k^2(x)(\psi_k^2(y)| h \p_y \phi_h|^2 +
(\psi_k'(y))^2  |\phi_h|^2) dV.
\end{align*}
Recall we are assuming the theorem is true for $k$, so we have
\[
\int_{[-2 h^{\eta_k}, 2 h^{\eta_k}]^2}| \phi_h|^2 dV = \O(h^{\eta_k}).
\]
Collecting terms, we have 
\[
I \leq C h^{1 - \eta_k} I + \O(h^{\eta_k}).
\]
Rearranging proves the claim since $\eta_k <1$.

We now use this Claim together with Lemma \ref{L:Sobolev} to control boundary terms.  
We  follow the proof in the $\delta = 2/3$ case.  We compute the
commutator, being very careful for ``lower order terms''.  Recalling
the definition \eqref{E:chi-def-k} of $\chi$:
\begin{align}
  \int_\Omega & ([-h^2 \Delta -1, \chi \p_x ] \phi_h) \phi_h dV \notag
  \\
  & = \int_\Omega ( (-2 \chi_x h^2 \p_x^2 - h \chi_{xx} h \p_x - 2
  \chi_{y} h \p_y h \p_x - h \chi_{yy} h \p_x ) \phi_h ) \phi_h dV.
  \label{E:comm-1001}
\end{align}
Let us examine each term separately.  We have
\begin{align}
  \int_\Omega & (-2 \chi_x h^2 \p_x^2\phi_h ) \phi_h dV \notag \\
  & = \int_{-r}^r \int_{\beta(y)}^r (-2 \chi_x h^2
  \p_x^2\phi_h ) \phi_h dx dy \notag \\
  & = 
\int_{-r}^r \int_{\beta(y)}^r (2 \chi_x |h
\p_x\phi_h|^2  dx dy \notag \\
& \quad + 
\int_{-r}^r \int_{\beta(y)}^r (2 h\chi_{xx} h
  \p_x \phi_h ) \phi_h dx dy
 \notag  \\
  & \quad - 
\int_{-r}^r  2 h\chi_x( h
\p_x\phi_h ) \phi_h |_{\beta(y)}^r  dy.  \label{E:IBP-1001}
\end{align}
The term in \eqref{E:IBP-1001} with $\chi_{xx}$ also shows up in
\eqref{E:comm-1001}.  We know that $\chi_{xx} = \O(h^{-2 \eta_{k+1}})$
and is supported on a set of radius $\sim h^{\eta_k}$, so our 
Claim \eqref{E:chi-k-deriv} gives
\begin{align*}
  \int_\Omega h \chi_{xx} (h \p_x \phi_h ) \phi_h dV  & = \O( h h^{-2
    \eta_{k+1}} h^{\eta_k})\\
  & = \O(1),
\end{align*}
since
\[
1 - 2 \eta_{k+1} + \eta_k = 1 - 2\left(1- \frac{1}{3(k+1)}\right) + 1
- \frac{1}{3k} = \frac{k-1}{3k(k+1)} \geq 0.
\]

For the two remaining terms in \eqref{E:IBP-1001}, we need to use the
support properties of $\chi_x$.  We have
\begin{align*}
  \chi_x & = h^{-\eta_{k+1}} \tchi'(x/h^{\eta_{k+1}})
  \tpsi^2(x/h^{\eta_k}) \tpsi^2 (y/h^{\eta_k}) \\
  & \quad + 2h^{-\eta_k} \tchi(x/h^{\eta_{k+1}})
  \tpsi'(x/h^{\eta_k})\tpsi(x/h^{\eta_k}) \tpsi^2 (y/h^{\eta_k}) .
\end{align*}
Recalling our function $\gamma(s) = \tchi'(s)$, we have
\[
\chi_x \geq h^{-\eta_{k+1}} \gamma(x/h^{\eta_{k+1}}) -
\O(h^{-\eta_k}),
\]
and let us stress again that the $\O(h^{-\eta_k})$ error term is
supported on scale $h^{\eta_k}$, so our Claim \eqref{E:chi-k-deriv} applies.  Hence we have
\[
\int_\Omega 2 \chi_x |h
\p_x\phi_h|^2  dV \geq h^{-\eta_{k+1}} \int_\Omega
\gamma(x/h^{\eta_{k+1}}) \gamma(y/h^{\eta_{k+1}}) | h \p_x \phi_h |^2
dV - \O(1).
\]

We now examine the boundary term in \eqref{E:IBP-1001}.  This is again
where we must be mindful of any differences between the case with or
without corners.  As in the previous steps in the proof, we will also be
using a commutant with the vector field $\rho \p_y$, where
\begin{equation}
\rho = \alpha'(x)\tchi(\beta(y)/h^{\eta_{k+1}}) \tpsi^2(x/h^{\eta_k})
\tpsi^2(y/h^{\eta_k}).
\label{E:rho-def-1001}
\end{equation}
The same cancellations of boundary terms will happen on the set where
$\rho_y = \chi_x$, which is for $-3h^{\eta_{k+1}} \leq x \leq 3
h^{\eta_{k+1}}$.  For $|x| \geq 3 h^{\eta_{k+1}}$, these functions do
not necessarily agree, but in this region both $\chi_x$ and $\rho_y$
are $\O(h^{-\eta_k})$ rather than $\O(h^{-\eta_{k+1}})$.  Further,
they are supported away from $x = 0$ so that we may further integrate
by parts on the boundary.  That is,
\begin{align*}
 & \int_{-r}^r   (2 h\chi_x h
\p_x\phi_h ) \phi_h |_{\beta(y)}^r  dy
\\
& = -\int_{-r}^r  2 h h^{-\eta_{k+1}}
\tchi'(x/h^{\eta_{k+1}}) \tpsi^2(\beta(y)/h^{\eta_k}) \tpsi^2 (y/h^{\eta_k}) h
\p_x\phi_h ) \phi_h (\beta(y),y) dy \\
&
\quad -\int_{-r}^r  4 h h^{-\eta_{k}}
\tchi(x/h^{\eta_{k+1}}) \tpsi'(\beta(y)/h^{\eta_k}) \tpsi(\beta(y)/h^{\eta_k})\tpsi^2 (y/h^{\eta_k}) h
\p_x\phi_h ) \phi_h (\beta(y),y)  dy.
\end{align*}
The cutoffs in the second term are supported away from $x = 0$, where
$\tchi = \pm 1$.  Let $\tau$ denote the tangent variable so that, as
above,
\[
\p_y \phi_h|_{\p \Omega} = \frac{\alpha'}{\kappa} \p_\tau \phi_h|_{\p
  \Omega}.
\]
Let
\[
\tzeta(y) = \tchi(\beta(y)/h^{\eta_{k+1}})  \tpsi'(\beta(y)/h^{\eta_k}) \tpsi^2
(y/h^{\eta_k}),
\]
and let $\zeta(\tau)$ denote $\tzeta$ in tangent coordinates,
so that $\p_\tau^m \zeta = \O(h^{-m\eta_k})$.  
Then
\begin{align*}
  \int_{-r}^r & 2 h h^{-\eta_{k}}
\tchi(x/h^{\eta_{k+1}}) \tpsi'(\beta(y)/h^{\eta_k})  \tpsi(\beta(y)/h^{\eta_k})\tpsi^2 (y/h^{\eta_k}) h
\p_x\phi_h ) \phi_h (\beta(y),y)  dy \\
& = \int_{\p \Omega}  h^{2 - \eta_k} \zeta(\tau)
\frac{\alpha'}{\kappa} \p_\tau ( | \phi_h|^2) d \tau \\
& = - \int_{\p \Omega}  h^{2 - \eta_k} \p_\tau (\zeta(\tau)
\frac{\alpha'}{\kappa})  | \phi_h|^2 d \tau \\
& = \O(h^{2 - 2 \eta_k} h^{\eta_k-1})
\\
& = \O(1),
\end{align*}
where we have used our Claim \eqref{E:chi-k-deriv} together with Lemma \ref{L:Sobolev} and that $\eta_k < 1$ for every
$k$.  Collecting terms, we have
\begin{align*}
  & \int_{-r}^r   (2 h\chi_x h
\p_x\phi_h ) \phi_h |_{\beta(y)}^r  dy
\\
& = -\int_{-r}^r  2 h h^{-\eta_{k+1}}
\tchi'(\beta(y)/h^{\eta_{k+1}}) \tpsi^2(\beta(y)/h^{\eta_k}) \tpsi^2 (y/h^{\eta_k}) h
\p_x\phi_h ) \phi_h (\beta(y),y) dy + \O(1).
\end{align*}

%

We continue with the other two terms in \eqref{E:comm-1001}.  We have
$\chi_y = \O(h^{-\eta_k})$ and $h \chi_{yy} = \O(h^{1-2 \eta_k}) =
\O(h^{-\eta_k})$, and we are integrating over a region of radius
$\sim h^{\eta_k}$, so using our Claim \eqref{E:chi-k-deriv} together with Lemma \ref{L:Sobolev} yet again,
\[
\int_\Omega ( (- 2
  \chi_{y} h \p_y h \p_x - h \chi_{yy} h \p_x ) \phi_h ) \phi_h dV = \O(1).
  \]

  We now use the vector field $\rho \p_y$ as in
  \eqref{E:rho-def-1001}.  All of the computations are similar, once
  again singling out the boundary terms which are supported near $x =
  0$ but where $\chi_x = \rho_y$ and summing as in the $\delta = 2/3$
  case, we get
  \begin{align*}
    \int_\Omega & ([-h^2 \Delta -1, \chi \p_x] \phi_h ) \phi_h dV +
    \int_\Omega ([-h^2 \Delta -1, \rho \p_y] \phi_h ) \phi_h dV \\
    & = 2 \int_\Omega \chi_x | h \p_x \phi_h |^2 dV + 2 \int_\Omega
    \rho_y | h \p_y \phi_h |^2 dV \\
    & \quad 
-\int_{-r}^r  2 h h^{-\eta_{k+1}}
\tchi'(\beta(y)/h^{\eta_{k+1}}) \tpsi(\beta(y)/h^{\eta_k}) \tpsi (y/h^{\eta_k}) h
\p_x\phi_h ) \phi_h (\beta(y),y) dy \\
& \quad   
+\int_{-r}^r  2 h h^{-\eta_{k+1}}
\tchi'(x/h^{\eta_{k+1}}) \tpsi(x/h^{\eta_k}) \tpsi (\alpha(x)/h^{\eta_k}) h
\p_y\phi_h ) \phi_h (x , \alpha(x)) dx + \O(1) \\
& \geq c_0 h^{-\eta_{k+1}} \int_\Omega \gamma(x/h^{\eta_{k+1}})
\gamma(y/h^{\eta_{k+1}}) | \phi_h |^2 dV - \O(1) \\
& \geq \frac{c_0}{4} h^{-\eta_{k+1}} \int_{\Omega\cap B(p_0, h^{\eta_{k+1}})}  | \phi_h |^2 dV - \O(1)
\end{align*}
for $c_0>0$ independet of $h$.

  Finally, we unpack the commutator as in the $\delta = 1/2$ case and
  use the claims and observations above to conclude that
  \[
  \int_\Omega ([-h^2 \Delta -1, \chi \p_x] \phi_h ) \phi_h dV +
    \int_\Omega ([-h^2 \Delta -1, \rho \p_y] \phi_h ) \phi_h dV =
    \O(1).
    \]
    This completes the proof in the case $p_0$ is not a corner.

     In
    the case $p_0$ is a corner, 
    we use the functions
    \[
    \rho_j = \alpha_j'(xy/\alpha_j(x)) \tchi(\beta_j(y)/h^{\eta_{k+1}}) \tpsi^2(x/h^{\eta_k})
\tpsi^2(y/h^{\eta_k}).
\]
We recall the following facts about the $\rho_j$s:  First, along $y =
\alpha_j$, $j = 1,2$, we have
\begin{equation}
  \label{E:rho-y-j-101}
\rho_j ( x , \alpha_j(x)) = \alpha_j'(x) \tchi(x/h^{\eta_{k+1}})\tpsi^2(x/h^{\eta_k})
\tpsi^2(\alpha_j(x)/h^{\eta_k})
 = \alpha_j' \chi(x,
\alpha_j(x)).
\end{equation}
Along $y = 0$, $\rho_j = 0$, since $\tchi$ is an odd function and
$\beta_j(0) = 0$ for $j = 1,2$.  Along
$y = \alpha_j$,
\begin{align*}
\p_y \rho_j(x, \alpha_j(x)) & = h^{-{\eta_{k+1}}} \alpha_j' (x)
\beta_j'(\alpha_j(x)) \tchi'(x/h^{\eta_{k+1}})
\tpsi^2(x/h^{\eta_k})\tpsi^2(\alpha_j(x)/h^{\eta_k}) + A_j + \O(h^{-\eta_k}) \\
& =  h^{-{\eta_{k+1}}} \tchi'(x/h^{\eta_{k+1}})
\tpsi^2(x/h^{\eta_k})\tpsi^2(\alpha_j(x)/h^{\eta_k}) + A_j + \O(h^{-\eta_k}) \\
& = \p_x \chi(x, \alpha_j(x)) + A_j + \O(h^{-\eta_k}).
\end{align*}
Here $A_j$ is the term we get when the derivative lands on the $\alpha_j'(xy/\alpha_j(x))$:
\[
A_j 
= (x/\alpha_j(x)) \alpha_j''(xy/\alpha_j(x)) \tchi( \beta_j(y) /h^{\eta_{k+1}}) \tpsi^2
(x/h^{\eta_k}) \tpsi^2 (y/ h^{\eta_k}).
\]

\begin{remark}
\label{R:corner-vanish-2}
We point out again that the implicit $\O(h^{-\eta_k})$ error term is due to differentiating the $\tpsi^2(y/h^{\eta_k})$ which is then supported away from $(0,0)$ on scale $\sim h^{\eta_k}$, hence does not see $\tchi(\beta_j(y)/h^{\eta_{k+1}})$.  
We single out the behaviour of $A_j$ because it is still singular due to the $\tchi(\beta_j/h^{\eta_{k+1}})$.
Indeed, we have
\[
\p_y^p A_j = \O(h^{-p \eta_{k+1}}),
\]
and for $k \geq 1$, $\p_y^k A_j$ is supported where $|y| \lesssim h^{\eta_{k+1}}$.
 We know  $\tchi (\beta_j(y)/h^{\eta_{k+1}})$ vanishes at $x = y = 0$ so that $A_j$ does as well.
\end{remark}

Finally, along $y =
0$, we have $\tchi'(0) = 1/2$ and $\tpsi'(0) = 0$, so that
\begin{align}
  \p_y \rho_j(x,0) &= h^{-\eta_{k+1}}  \tchi'(0) \tpsi^2(x/h^{\eta_k}) \tpsi^2(0) + \O(h^{-\eta_k})\\
  & = (h^{-\eta_{k+1}}/2) \tpsi(x/\epsilon) + \O(h^{-\eta_k}).\label{E:rho-y-0-2}
  \end{align}

The only new element in the proof at this point is in estimating the boundary terms involving the $A_j$, so let us very quickly see what happens.  
 Continuing, using Remark \ref{R:corner-vanish-2}:
  \begin{align}
  \int_{\tau = 0}^r  h \tA_j(\tau) (h \p_\tau \phi_h) \bar{\phi}_j d \tau & = \frac{h^2}{2} \int_0^r \tA_j \p_\tau |\phi_h|^2 d \tau \notag \\
  & = -\frac{h^2}{2} \int_0^r (\p_\tau \tA_j) | \phi_h|^2 d \tau \notag \\ 
  & = \O(h^{1-\eta_{k+1} + \eta_k}) \notag \\
  & = \O(1) \label{E:Aj-control-3}
  \end{align}
  from Lemma \ref{L:Sobolev} with $h^{\eta_k}$.

  \end{proof}


 \bibliographystyle{alpha}
 \bibliography{HC-bib}

\def\cprime{$'$} \def\cftil#1{\ifmmode\setbox7\hbox{$\accent"5E#1$}\else
  \setbox7\hbox{\accent"5E#1}\penalty 10000\relax\fi\raise 1\ht7
  \hbox{\lower1.15ex\hbox to 1\wd7{\hss\accent"7E\hss}}\penalty 10000
  \hskip-1\wd7\penalty 10000\box7}
\begin{thebibliography}{Han15}

\bibitem[CT23]{ChTo-non-con-smooth}
Hans Christianson and John~A. Toth.
\newblock Small-scale mass estimates for laplace eigenfunctions: interior
  estimates and compact $c^\omega$ riemann surfaces with concave boundary.
\newblock {\em in preparation}, 2023.

\bibitem[Gri11]{Grisvard-book}
Pierre Grisvard.
\newblock {\em Elliptic problems in nonsmooth domains}, volume~69 of {\em
  Classics in Applied Mathematics}.
\newblock Society for Industrial and Applied Mathematics (SIAM), Philadelphia,
  PA, 2011.
\newblock Reprint of the 1985 original [ MR0775683], With a foreword by Susanne
  C. Brenner.

\bibitem[Han15]{Han}
Xiaolong Han.
\newblock Small scale quantum ergodicity in negatively curved manifolds.
\newblock {\em Nonlinearity}, 28(9):3263--3288, 2015.

\bibitem[HR16]{HR}
Hamid Hezari and Gabriel Rivi\`ere.
\newblock {$L^p$} norms, nodal sets, and quantum ergodicity.
\newblock {\em Adv. Math.}, 290:938--966, 2016.

\bibitem[Sog16]{So}
Christopher~D. Sogge.
\newblock Localized {$L^p$}-estimates of eigenfunctions: a note on an article
  of {H}ezari and {R}ivi\`ere.
\newblock {\em Adv. Math.}, 289:384--396, 2016.

\bibitem[Tat98]{Ta}
Daniel Tataru.
\newblock On the regularity of boundary traces for the wave equation.
\newblock {\em Ann. Scuola Norm. Sup. Pisa Cl. Sci. (4)}, 26(1):185--206, 1998.

\bibitem[TZ09]{TZ1}
John~A. Toth and Steve Zelditch.
\newblock Counting nodal lines which touch the boundary of an analytic domain.
\newblock {\em J. Differential Geom.}, 81(3):649--686, 2009.

\end{thebibliography}

\end{document}